\documentclass[11pt,letterpaper]{amsart}
\usepackage{amssymb,mathtools,amsthm}
\usepackage[T1]{fontenc}
\usepackage[utf8]{inputenc}
\usepackage{microtype}
\usepackage{enumitem}
\usepackage{caption}
\usepackage{dsfont}

\setenumerate[1]{label={\arabic*.},ref={\arabic*}}

\mathtoolsset{showonlyrefs}

\usepackage{tikz}
\usetikzlibrary{arrows.meta}
\usepackage{subcaption}

\usepackage{hyperref}
\definecolor{colorlinks}{RGB}{0, 24, 168}
\definecolor{colorcites}{RGB}{124, 10, 2}
\hypersetup{
	colorlinks=true,
	linkcolor=colorlinks,
	citecolor=colorcites,
	urlcolor=colorlinks,
	pdfborder={0 0 0}
}

\makeatletter
\def\paragraph{\medskip \@startsection{paragraph}{4}%
  \z@\z@{-\fontdimen2\font}%
  {\normalfont\bfseries}}
\makeatother

\usepackage{graphicx} 
\usepackage{amsmath}
\usepackage{amssymb}
\usepackage{amsthm}
\usepackage{bbm}

\usepackage{empheq}
\usepackage{enumitem}
\usepackage{color}
\usepackage{version}
\usepackage{graphicx}
\usepackage{amsfonts}
\usepackage{dsfont}
\usepackage{a4wide}
\usepackage{overpic}
\usepackage{blindtext}
\usepackage{xcolor}
\usepackage{mathabx}
\usepackage{mathtools}

\theoremstyle{plain}
\newtheorem{Thm}{Theorem}[section]
\newtheorem{Lem}[Thm]{Lemma}
\newtheorem{Cor}[Thm]{Corollary}

\theoremstyle{definition}
\newtheorem{Def}[Thm]{Definition}

\theoremstyle{remark}
\newtheorem{Rem}[Thm]{Remark}

\newcommand{\cR}{\mathcal{R}}
\newcommand{\cM}{\mathcal{M}}

\newcommand{\cD}{\mathcal{D}}

\newcommand{\cP}{\mathcal{P}}

\newcommand{\Hexa}{\mathbb{H}}
\newcommand{\Triangular}{\mathbb{T}}

\newcommand{\R}{\mathbb{R}}

\newcommand{\N}{\mathbb{N}}
\newcommand{\Z}{\mathbb{Z}}

\newcommand{\E}{\mathbb{E}}

\newcommand{\eps}{\varepsilon}
\newcommand{\flip}{\mathrm{flip}}

\newcommand{\be}{\begin{equation}}
\newcommand{\ee}{\end{equation}}

\newcommand{\Omfp}{\Omega_{\mathrm{fp}}}
\newcommand{\Npath}{N_{\infty}}

\newcommand{\tri}{\mathcal{T}_3}
\newcommand{\tricoarse}{\mathsf{Tri}}
\newcommand{\loops}{\mathrm{Loops}}
\newcommand{\paths}{\mathrm{Paths}}

\newcommand{\nocrazyloop}{\mathcal{L}}
\newcommand{\nocrazypath}{\mathcal{P}}
\newcommand{\cluster}{\mathcal{C}^*}
\newcommand{\clusterbis}{\mathcal{D}^*}
\newcommand{\clusters}{\mathrm{Clusters}^*}
\newcommand{\tomega}{\tilde{\randomloop}}
\newcommand{\tgamma}{\tilde{\gamma}}

\newcommand{\tg}{\tilde{g}}
\newcommand{\area}{\mathrm{Area}}
\newcommand{\glue}{\mathcal{G}}
\newcommand{\infclusters}{\clusters_{\infty}}
\newcommand{\tcluster}{\tilde{\cluster}}
\newcommand{\flippable}{\mathcal{F}_{\fzero}}

\newcommand{\Np}{N}
\newcommand{\conditions}{\mathcal{S}}

\newcommand{\fzero }{\mathsf{0}}

\newcommand{\Pdimer}{\mathsf{Dim}}
\newcommand{\Pddimer}{\mathsf{Double}\Pdimer}
\newcommand{\Ploop}{\mathsf{Loop}}
\newcommand{\contradict}{\mathsf{Paral}}
\newcommand{\randomloop}{w}

\renewcommand{\dim}{m}
\newcommand{\dimbis}{\dim'}
\newcommand{\randomdim}{M}
\newcommand{\ddim}{d}
\newcommand{\randomddim}{D}

\newcommand{\NewtonPolygon}{N}

\newcommand{\dualpath}{\gamma^*}
\newcommand{\concatene}{}
\newcommand{\tree}{\mathsf{T}_\infty}

\newcommand{\norm}[1]{\left\lVert#1\right\rVert}

\newcommand{\1}{\mathbbm{1}}

\title{On loops in the complement of dimers}
\author{Alexander Glazman}
\address{Department of Mathematics, University of Innsbruck, Technikerstr. 13, A-6020 Innsbruck, Austria}
\email{alexander.glazman@uibk.ac.at}

\author{Lucas Rey}
\address{CEREMADE, Universit{\'e} Paris-Dauphine, Université PSL, CNRS, 75016 Paris, France and DMA, {\'E}cole normale sup{\'e}rieure, Universit{\'e} PSL, CNRS, 75005 Paris, France}
\email{lucas.rey@inrae.fr}

\begin{document}
\mathtoolsset{showonlyrefs}
\maketitle

\begin{abstract}
	We consider ergodic translation-invariant Gibbs measures for the dimer model (i.e., perfect matchings) on the hexagonal lattice.
	The complement of a dimer configuration is a fully-packed loop configuration: each vertex has degree two.
	This is also known as the loop~$O(1)$ model at~$x=\infty$.
	We show that, if the measure is non-frozen, then it exhibits either infinitely many loops around every face or a unique bi-infinite path.
	
	Our main tool is the flip (or XOR) operation: if a hexagon contains exactly three dimers, one can replace them by the other three edges.
	Classical results in dimer theory imply that such hexagons appear with a positive density. 
	To some extent, this replaces the finite-energy property and allows to make use of tools from the percolation theory, in particular the Burton--Keane argument, to exclude existence of more than one bi-infinite path.
\end{abstract}



\section{Introduction}

The loop~$O(n)$ model is defined on collections of non-intersecting cycles (called {\em loops} in this context) on the hexagonal lattice.
The probability of a configuration has two parameters~$n,x>0$ and depends on the number of loops and edges contained in them.
The model has attracted considerable attention due to conjectured conformal invariance and connections to a number of other classical models.
For the background on the loop~$O(n)$ model, we refer the reader to~\cite{PelSpi17} for a general survey and to the introduction of~\cite{GlaHarZel25} for a description of recent progress.

In the current work, we restrict our attention to the case~$n=1$.
The loop~$O(1)$ model at~$x=1$ describes boundaries of clusters in the critical site percolation on the triangular lattice.
These boundaries were proven to converge to a random fractal limit called the Conformal Loop Ensemble (CLE) with parameter~$\kappa = 6$ in spectacular works developing the method of Smirnov's parafermionic observable~\cite{Smi01,CamNew06} (see also~\cite{KhrSmi21}).
When~$x\neq 1$, the loop~$O(1)$ model describes the boundaries between clusters of pluses and minuses (so-called {\em domain walls}) in the ferromagnetic ($x<1$) or antiferromagnetic ($x>1$) Ising model on the triangular lattice.
The ferromagnetic Ising model has been extensively studied, in particular Smirnov et al proved conformal invariance at the critical point~$x_c=1/\sqrt{3}$ in~\cite{Smi10,CheSmi11}, showing that the associated loop ensemble converges to~$\mathrm{CLE}_3$~\cite{CheDumHonKemSmi14,BenHon19}.
When~$1/\sqrt{3}<x<1$, the loop~$O(1)$ model satisfies the Russo--Seymour--Welsh estimates, meaning that it exhibits macroscopic loops at every scale (follows from e.g.~\cite{Tas16}); notably, conformal invariance remains open in this case.
It is classical that the phase transition in the ferromagnetic Ising model is sharp~\cite{AizBarFer87} (see also~\cite{DumTas15}); thus, the loop~$O(1)$ model exhibits exponential decay when~$0<x<1/\sqrt{3}$.

The antiferromagnetic ($x>1$) Ising model on the triangular lattice is far less understood.
When~$x>\sqrt{3}$, no results are available while in the regime~$1<x<\sqrt{3}$, it was only recently established that the model exhibits either infinitely many loops around every face or a bi-infinite path~\cite{CraGlaHarPel20}, using Newman's representation~\cite{New90}. 
In the current work, we prove a similar result for the limiting case~$x=\infty$ that we define as the uniform measure on fully-packed loop configurations.

More precisely, let~$\Hexa$ be the hexagonal lattice whose faces are centered at~$\{n+m \cdot e^{\pi i /3} \colon n,m\in\Z^2\}$.
Denote its sets of vertices, edges, and faces by~$V(\Hexa)$, $E(\Hexa)$, and~$F(\Hexa)$ respectively.
For $V \subset V(\Hexa)$, define the set of edges {\em incident} to~$V$ by~$E_{V} := \{ uv \in E(\Hexa)\colon u \in V\}$ and the {\em edge-boundary} of~$V$ by~$\partial E_V := \{uv \in E(\Hexa) \colon u \in V, v \in V(\Hexa) \setminus V\}$.
A {\em fully-packed} loop configuration on~$V$ is any subset~$\randomloop \subset E_{V}$ such that each vertex~$u\in V$ has degree two in~$\randomloop$ (see Fig.~\ref{fig:fully-packed}).
The set of all fully-packed loop configurations on~$V$ is denoted by~$\Omfp(V)$.
When~$V=V(\Hexa)$, we use the notation~$\Omfp :=\Omfp(V(\Hexa))$ for brevity.
We also introduce the set of configurations with boundary conditions~$\eta\in \Omfp$:
\[
    \Omfp(V,\eta) = \{\randomloop \in \Omfp(V) \colon \randomloop \cap \partial E_{V} = \eta \cap \partial E_{V}\}.
\]

\begin{figure}
	\begin{center}
		\includegraphics[width=0.4\textwidth]{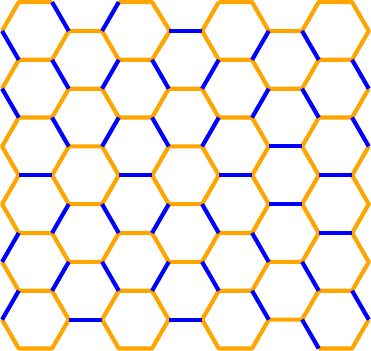}
	\end{center}
	\caption{An example of a fully packed loop configuration on the hexagonal lattice: loops and infinite paths are shown in orange, dimers (in the complement) are in blue.}
	\label{fig:fully-packed}
\end{figure}

\begin{Def}[Loop~$O(1)$ model at~$x=\infty$]\label{def:loop-meas-fully-packed}
	Let $V \subset V(\Hexa)$ be finite.
	The loop $O(1)$ model on $V$ at $x = \infty$ with boundary conditions $\eta \subset \Omfp$ is the uniform measure on $\Omfp(V, \eta)$.
	A measure $\Ploop$ on $\Omfp$ is an \emph{infinite-volume Gibbs measure} for the loop $O(1)$ model at~$x = \infty$ if it satisfies the following Dobrushin--Lanford--Ruelle (DLR) conditions:
	for any finite~$V \subset V(\Hexa)$ and~$\eta \in \Omfp$, 
    the conditional distribution of $\Ploop$ given that the configuration on~$\partial E_{V}$ coincides with~$\eta$, restricted to~$E_{V}$, is the uniform measure on~$\Omfp(V,\eta)$. 
\end{Def}

We formulate our results in the framework of infinite-volume Gibbs measures, which allows us to work in a translation-invariant setting. In this setting, a Gibbs measure is \emph{ergodic} if any translation-invariant event has probability $0$ or $1$.
Certain pathological examples must be excluded.
For instance, one may consider a Dirac measure on the subset of~$E(\Hexa)$ consisting of all non-vertical edges.
\begin{Def}\label{def:frozen}
    A translation-invariant Gibbs measure $\Ploop$ is \emph{frozen} if there exists $e \in E(\Hexa)$ such that all translates of $e$ are almost surely present under $\Ploop$. When this is not the case, we say that $\Ploop$ is \emph{non-frozen}.
\end{Def}

Let~$\mathfrak{G}$ denote the set of all translation-invariant ergodic non-frozen Gibbs measures for the loop~$O(1)$ model at~$x=\infty$.
\begin{Rem}
    We will explain using dimer tools that this set is non-empty, see Remark~\ref{rem:toroidal:exhaustion} and Appendix~\ref{app:dimers}. 
    For example, it contains the weak limit of the uniform measures on $\Omfp(V_n)$ with periodic boundary conditions, where $(V_n)_{n \geq 1}$ is an exhaustion of $V(\Hexa)$. 
\end{Rem}
The following theorem is the main result of the current paper.

\begin{Thm}\label{thm:k:inf:2}
	Let~$\Ploop \in \mathfrak{G}$. Then, either $\Ploop$-almost surely there are infinitely many loops surrounding every face, or $\Ploop$-almost surely there exists a unique bi-infinite path.
\end{Thm}

It is widely believed that, barring some pathological examples, infinite-volume measures in two dimensions do not exhibit any bi-infinite paths.
This has so far been proven only in a limited number of cases: for the Ising model in classical works by Aizenman~\cite{Aiz80} and Higuchi~\cite{Hig81} and more recently for the Potts model when $T<T_c$~\cite{CoqDumIofVel14} and then at~$T_c$~\cite{GlaMan23} (the case~$T>T_c$ being trivial).
The work~\cite{GlaMan23} also applies to the loop~$O(n)$ model and rules out bi-infinite paths when~$n\geq 1$ and~$nx^2 \leq 1$.
That argument relies on the non-coexistence theorem of Zhang and Sheffield~\cite{She05} (see also~\cite[Theorem~1.5]{DumRaoTas19}), which in turn uses the positive correlation (FKG) inequality for a suitable spin representation.
The loop~$O(1)$ model naturally lacks this property when~$x>1$ (including~$x=\infty$) which prevents us from ruling out the existence of a unique bi-infinite path.
After the current work came out, bi-infinite paths have been also ruled out for all~$n\in [1,2], x\in [1/\sqrt{2},1]$~\cite{GlaHarZel25}.
The approach relies on a {\em quenched} FKG that also does not extend to~$x=\infty$.

In the absence of positive correlation inequalities, we rely on the following correspondence: the complement of any fully-packed loop configurations on~$\Hexa$ is a perfect matching, also called a {\em dimer} configuration.
The dimer model is by now well understood due to its determinantal structure discovered by Kasteleyn~\cite{Kas61}.
Notably, the associated height function introduced by Thurston~\cite{Th90} was proven to converge to the Gaussian Free Field (GFF) by Kenyon~\cite{Ken01,Ken02}.
We provide further details on this height function in Section~\ref{subsec:height:function}.
We already note that convergence to the GFF does not tell anything about these loops, since they separate heights of different parity.
Although mixing properties are not used in our arguments, we mention that they were established for the dimer model by Laslier~\cite{Las19}.

Finally, we explain how our result fits into the predicted phase diagram of the loop~$O(n)$ model.
It is believed that the loop~$O(n)$ model admits a conformally invariant scaling limit for all~$n\in [0,2]$ and~$x\geq x_c(n)=1/\sqrt{2+\sqrt{2-n}}$~\cite{Nie82,KagNie04}.
Apart from the two special points~$n=1,x=1$ and~$n=1,x=1/\sqrt{3}$, corresponding respectively to critical site percolation and the critical Ising model on the triangular lattice, this conjecture remains completely open.
Our result provides evidence that, barring the presence of a unique bi-infinite path which we are unable to exclude, the loop~$O(1)$ model at~$x=\infty$ may exhibit a non-trivial scaling limit.
A natural next step would be to establish Russo--Seymour--Welsh estimates; which would imply the existence of macroscopic loops.

It is expected that the scaling limit at~$n=1$ should be CLE(4) at~$x=\infty$ and CLE(6) for any finite~$x>1/\sqrt{3}$~\cite{BloNie89}.
The emergence of different behavior in the fully packed case~$x=\infty$ is apparently linked to the fact that the hexagonal lattice is bipartite; see the work of Jacobsen~\cite{Jac99} and references therein.
Spectacular progress has been achieved in the study of a fully packed loop~$O(n)$ model on (the dual graph to) random planar triangulations: the asymptotics of loop lengths have very recently been established up to a constant~\cite{BerDaS25}. 
At~$n=1$, the critical exponents are in agreement with the scaling limit being CLE(6).
The case of random {\em bipartite} planar triangulations remains open.
We point out that the fully packed loop~$O(2)$ model in this case corresponds to proper four-colorings of the vertices.

\paragraph{Our approach and tools.}
We combine fundamental results from the dimer theory with a percolation approach to prove our main result.

The loop $O(n)$ at $x = \infty$ is very rigid and, in particular, does not have the insertion tolerance property. 
Nonetheless, we manage to change the topology of bi-infinite paths via a bounded number of local moves: we show this in the swapping Lemma~\ref{lem:swap}, which is our main technical result.
The proof of this lemma relies on the description of the set of ergodic Gibbs measures on $\Z^2$-periodic graphs~\cite{She05, KenOkoShe06} together with an explicit computation of edge probabilities.
We use this swapping lemma together with a modified version of the Burton-Keane argument of~\cite{BurKea89} to prove that there is at most one bi-infinite path.

In the case where there is almost surely no bi-infinite path, we further combine percolation and dimer techniques to prove that there are infinitely many loops surrounding every point.
From the percolation theory, we borrow the XOR-argument of~\cite{CraGlaHarPel20}.
We combine this argument with the phase diagram of the double dimer model established in~\cite{She05,KenOkoShe06} and, more precisely with the fact that, in the rough phase, there are almost surely infinitely many double dimer loops surrounding every point.

\paragraph{Acknowledgments.} 
We would like to thank Benoit Laslier and Béatrice de Tilière for many fruitful discussions, especially regarding mixing properties of the dimer model, the antiferromagnetic Ising model and parts of the Appendix.
We also thank anonymous referees for their insightful comments that led to numerous improvements.
This research was funded in whole or in part by the Austrian Science Fund (FWF) 10.55776/P34713 and 10.55776/AST1405524.

\section{Notation and the flip operation}\label{sec:notation}

\subsection{Graph notation}
Given a finite set~$A$, we denote its cardinality by~$|A|$.
For~$z\in \R^2$, we denote the Euclidean norm of~$z$ by~$\norm{z}$.
For $r >0$, define $B(z,r):=\{w\in\R^2 \colon \norm{w-z}<r\}$ and~$\partial B(z,r):=\{w\in\R^2 \colon \norm{w-z}=r\}$.
When~$z=0$, we use the shorthand notation~$B(r):=B(0,r)$.

The triangular lattice $\Triangular$ is the dual graph to~$\Hexa$: its vertex~set is $V(\Triangular) = \{n+m \cdot e^{\pi i /3} \colon n,m\in\Z^2\}$ and its edge set $E(\Triangular)$ consists of pairs of vertices at Euclidean distance $1$.
The faces $x_1x_2x_3 \in F(\Triangular)$ of $\Triangular$ are equilateral triangles.
For a face~$f$ of~$\Hexa$, we denote by~$\partial f$ the set of edges in its boundary and call each of them a {\em side} of~$f$.
We denote by $F(\Hexa)$ the set of faces of $\Hexa$, and by $\fzero $ the face of $F(\Hexa)$ centered at the origin $(0,0)$.
When two distinct faces $f, g \in F(\Hexa)$ share a common edge $e \in \partial f \cap \partial g$, we write $f \sim g$ and denote this edge by $(fg)^{\star}:=e$. 
For $f, g \in F(\Hexa)$, we denote by~$\norm{f-g}$ the Euclidean distance between the centers of~$f$ and~$g$ .
Viewing faces of~$\Hexa$ as open hexagons in~$\R^2$, we write for $X \subset \R^2$,
\[
    F(\Hexa) \cap X := \{f \in F(\Hexa), f \cap X \neq \emptyset\}.
\]

For $x,y \in V(\Hexa)$, a \emph{path} $\gamma$ from $x$ to $y$ is a sequence of distinct vertices $x = x_0, x_1, \dots, x_n = y$, where~$x_ix_{i+1}\in E(\Hexa)$ for any~$i=0,\dots,n-1$. 
For $X \subset \R^2$, we say that $\gamma$ is a path in $X$ and we write $\gamma \sqsubset X$ if $x_i \in X$ for all $1 \leq i \leq n-1$. 
Note that we do not require the endpoints to belong to $X$, hence this non-standard notation. 
Similarly, for $f,g \in F$, a \emph{dual path} $\dualpath$ from $f$ to $g$ is a path $f=f_0, \dots, f_n = g$ in the dual graph.
We write $\dualpath \sqsubset X$ if $f_i \in X$ for all $1 \leq i \leq n-1$ . 
The \emph{concatenation} of two paths (resp. dual paths) sharing an endpoint is denoted by $\gamma_1 \concatene \gamma_2$ (resp.~$\dualpath_1 \concatene \dualpath_2$).

\subsection{Flip operation}

For two sets~$A,B$, we denote their symmetric difference (or {\em XOR}) by~$A \oplus B$.
The following combinatorial operation is one of our key tools.
\begin{Def}[Flips]\label{def:flip}
	For any~$f\in F(\Hexa), E\subset E(\Hexa)$, define
	\[
		\flip_f(E) := 
		\begin{cases}
			E \oplus \partial f & \text{ if } |E\cap \partial f| = 3,\\
			E & \text{ otherwise.}
		\end{cases}  
	\]
	This defines an involution~$\flip_f$ on subsets of~$E(\Hexa)$, which we call the {\em flip of $f$}.
    For faces $f_1, \dots f_n \in F(\Hexa)$, we write $\flip_{f_n, \dots, f_1}(E) := \flip_{f_n} \circ \cdots \circ \flip_{f_1}(E)$ for the composition of the flips.
\end{Def}

\subsection{Clusters}\label{sec:notation:clusters}
We now introduce some additional notation. For $\randomloop \in \Omfp$, we define the dual configuration $\randomloop^* \subset E(\Triangular)$ as follows:
a dual edge~$e^* \in E(\Triangular)$ belongs to~$\randomloop^*$ if and only if the corresponding primal edge~$e$ does not belong to~$\randomloop$.
\begin{Def}[Dual clusters]
	A dual cluster $\cluster \subset F(\Hexa)$ of $\randomloop \in \Omfp$ is the vertex set of a connected components of~$\randomloop^*$.
	For~$f\in F(\Hexa)$, we denote by~$\cluster_\randomloop(f)$, or $\cluster(f)$ when there is no risk of confusion, the dual cluster containing~$f$ .
	We denote by $\clusters(\randomloop)$ the set of all dual clusters in~$\randomloop$, and by~$\infclusters(\randomloop)$ the set of all infinite dual clusters.
	
	Two {\em distinct} clusters $\cluster_1, \cluster_2 \in \clusters(\randomloop)$ are $\emph{adjacent}$ if there exist two adjacent faces~$f,g\in F(\Hexa)$ such that $f \in \cluster_1$ and $g \in \cluster_2$. In this case, we write $\cluster_1 \sim \cluster_2$.
\end{Def}

We denote by $\paths(\randomloop)$ the set of infinite paths in~$\randomloop$; these are precisely interfaces between adjacent infinite dual clusters.

\section{Input from the dimer model}

\subsection{Dimers on the hexagonal lattice}\label{subsec:dimers}

\begin{Def}[Dimer configuration]\label{def:dimer-conf}
	For $V \subset V(\Hexa)$, a \emph{dimer configuration} (or a \emph{perfect matching}) on~$V$ is any~$m\subset E_{V}$ such that each vertex in~$V$ has degree~$1$ in~$m$.
	We denote the set of all dimer configurations on~$V$ by~$\cM(V)$.
	If~$V=V(\Hexa)$, we use the shorthand notation~$\cM:=\cM(V(\Hexa))$.
	For $V \subset V(\Hexa)$, $\dimbis \in \cM$, we define the set of dimer configurations with \emph{boundary conditions} $\dimbis$ as
	\[
	    \cM(V, \dimbis) = \{\dim \in \cM(V) \colon \dim \cap \partial E_{V} = \dimbis \cap \partial E_{V}\}.
	\]
\end{Def}

\begin{Def}[Dimer measure]\label{def:dimers-meas}	
	When $V$ is finite, we denote by $\Pdimer_V$ the uniform probability measure on $\cM(V)$ and by $\Pdimer_{V}^{\dimbis}$ the uniform probability measure on $\cM(V, \dimbis)$. A probability measure $\Pdimer$ on $\cM$ is a \emph{Gibbs measure} if it satisfies the DLR condition: for any finite $V \subset V(\Hexa)$ and~$\dimbis \in \cM$, conditioned on the configuration on~$\partial E_{V}$ coinciding with~$\dimbis$, the restriction of $\Pdimer$ to~$E_{V}$ is~$\Pdimer_V(\,\cdot \, |\,\dimbis \cap \partial E_{V})$.

	We say that~$\Pdimer$ is \emph{translation-invariant} if it is invariant under the action of the translations by $1$ and $e^{\pi i/3}$. 
    It is \emph{ergodic} if every translation-invariant event has probability $0$ or $1$. 
\end{Def}

\subsection{Loop $O(1)$ model and dimers.}\label{subsec:loop&dimers}
Let~$V\subset V(\Hexa)$.
By definition, for any vertex~$v\in V$, the set $E_{V}$ contains all three edges incident to~$v$.
This leads to the following correspondence: the complement of a dimer configuration on~$V$ is a fully packed loop configuration on~$V$, and vice versa.
The same holds also under any boundary conditions. Since the DLR conditions of Definitions \ref{def:loop-meas-fully-packed} and \ref{def:dimers-meas} coincide under this bijection, \emph{Gibbs measures for the loop $O(1)$ model at $x = \infty$ are in one-to-one correspondence with Gibbs measures for dimers}.

This bijection is also compatible with the flip operation: if $\randomloop = E(\Hexa) \setminus \dim$, then
\[
    \forall f\in F(\Hexa), \quad \flip_f(\randomloop) = E(\Hexa) \setminus \flip_f(\dim).
\]
\begin{Rem}\label{rem:parity}
    This bijection is a special case of the \emph{Fisher correspondence} \cite{Fis66} which is a $1$-to-$2$ measure-preserving map between the loop $O(1)$ model on a graph $G$ with weight~$x>0$ and a certain weighted dimer model on an associated decorated graph $G_F$.
    In the special case~$x = \infty$, some of the edge weights on $G_F$ vanish and one may take $G_F = G$.
\end{Rem}

\subsection{Loop $O(1)$ model and height functions.}\label{subsec:height:function}

The reader can skip this subsection on a first reading since the definition of the height function is not necessary until Appendix~\ref{app:dimers}.
We include it here since it can help the reader familiar with the topic to understand better the link between dimers and the loop $O(1)$ model.

Dimer configurations can be related to two-dimensional surfaces through the concept of \emph{height functions}, introduced by Thurston~\cite{Th90}. 
We first partition $V(\Hexa)$ into two sets of vertices~$WV$ (white vertices) and~$BV$ (black vertices) as follows:
\[
	WV:=\{n+ \tfrac13 + (m+\tfrac13) \cdot e^{\pi i /3} \colon n,m\in\Z\}
	\quad ; \quad
	BV:=\{n - \tfrac13 + (m+\tfrac23) \cdot e^{\pi i /3} \colon n,m\in\Z\}.
\]
\begin{Def}[Height function]\label{def:height:function}
	For~$\dim\in \cM$, define $h_{\dim}\colon F(\Hexa) \to \Z$ such that $h_{\dim}(\fzero )=0$ and, for any two adjacent faces~$f,g\in F(\Hexa)$,
	\[
		h_{\dim}(g)-h_{\dim}(f) = 
		\begin{cases}
			1 - 3\mathbbm{1}_{\partial f \cap \partial g \in \dim} & \text{if } \tfrac{1}{\sqrt{3}}(f\cdot e^{\pi i/6} + g\cdot e^{-\pi i/6}) \in BV,\\
			-1 + 3\mathbbm{1}_{\partial f \cap \partial g \in \dim} & \text{if } \tfrac{1}{\sqrt{3}}(f\cdot e^{\pi i/6} + g\cdot e^{-\pi i/6}) \in WV,
		\end{cases}
	\]
	where~$f,g$ are identified with points in the complex plane placed at their centers.
\end{Def}
It can be checked that this definition is consistent, and that dimer configurations and height functions are in bijection; see e.g. the review \cite{SalTom97}. A Gibbs measure on dimers induces a probability measure on the set of height functions via this bijection.

This definition of the height function provides another perspective on the correspondence of Section~\ref{subsec:loop&dimers}.
The interested reader can check the following fact:
$\randomloop = E(\Hexa) \setminus \dim$ is exactly the set of edges separating two faces~$f$ and~$g$ with different parity of the height function, i.e., $h_{\dim}(f) \neq h_{\dim}(g) \mod 2$.
Equivalently, the height function $h_{\dim}$ has constant parity on a dual cluster $\cluster \in \cluster(\randomloop)$ and opposite parity on two adjacent dual clusters $\cluster_1 \sim \cluster_2 \in \clusters(\randomloop)$.

\subsection{Translation-invariant ergodic non-frozen Gibbs measures.}

In this section, we state several standard results about the dimer model on the hexagonal lattice.
They can be derived from the general theory in a straightforward way. 
We provide some details in Appendix~\ref{app:dimers}.

Recall the definition of a frozen measure $\Ploop$ from Definition \ref{def:frozen}. By the correspondence of Section \ref{subsec:loop&dimers}, a translation-invariant measure $\Pdimer$ is frozen if and only if $\Pdimer(e)=0$ for all translates of some $e \in E$. 
The set of translation-invariant ergodic dimer Gibbs measures on bipartite $\Z^2$-periodic graphs can be described and parameterized explicitly. These are deep results from \cite{She05} and \cite{KenOkoShe06} which we recall in the restricted setting of the hexagonal lattice in Appendix~\ref{app:dimers}.

Theorem \ref{thm:toroidal:exhaustion} states that translation-invariant ergodic dimer Gibbs measures can be parametrized in a natural way by the triangle in~$\R^2$ with vertices $(0,0),(0,1),(1,0)$. The non-frozen measures correspond exactly to the interior of the triangle.
Theorem~\ref{thm:KOS} expresses these measures explicitly.

\begin{Rem}\label{rem:toroidal:exhaustion}
1) For $k \in \N$, let $\cM_k$ denote the set of dimer configurations which are invariant to the translations by $k$ and $k\cdot e^{\pi i/3}$. By Theorem \ref{thm:KOS} and Corollary \ref{cor:frozen}, the weak limit $\Pdimer^{\mathrm{per}}$ of the uniform measures on $\cM_k$ belongs to the set of translation-invariant non-frozen ergodic Gibbs measures.

2) Let $\Ploop^{\mathrm{per}}$ be the image of $\Pdimer^{\mathrm{per}}$ under the bijection between dimers and fully-packed loops. In other words, $\Ploop^{\mathrm{per}}$ is the weak limit of loop~$O(1)$ measures on tori under periodic boundary conditions (we have not defined them in detail). By Section \ref{subsec:loop&dimers} and 1), it is a translation-invariant ergodic non-frozen Gibbs measure for the loop $O(1)$ model at~$x = \infty$.

3) We do not know how to prove such weak convergence of loop $O(1)$ measures with periodic boundary conditions when~$1<x<\infty$.
    The case~$x\leq 1$ corresponds to the ferromagnetic Ising model and the seminal works of Aizenman~\cite{Aiz80} and Higuchi~\cite{Hig81}.
    In this case, the loop~$O(1)$ model has a unique Gibbs measure that can be obtained as the limit of the measures on tori under periodic boundary conditions.

4)
The infinite volume limit of the loop $O(1)$ measure at $x = \infty$, or equivalently of the uniform dimer measure, on a sequence of planar graphs is much more subtle.
This is related to the so-called \emph{limit shape phenomenon}, which has attracted a lot of attention. 
The local convergence of the uniform measure was conjectured by Cohn, Kenyon and Propp in 2001~\cite[Conjecture~13.5]{CohKenPro01} and solved only recently in its full generality by Aggarwal in \cite{Agg23}.
Under certain conditions on the boundary, the local limit is one of the ergodic translation-invariant Gibbs measures of Theorem~\ref{thm:KOS}, whose slope depends on the domain and on the details of the boundary.
To our knowledge, these boundary conditions (while being natural from the dimer perspective) have no reasonable interpretation for the loop $O(1)$ model.
\end{Rem}

At the center of our approach is invariance of Gibbs measures to flips (see Definition \ref{def:flip}):

\begin{Lem}\label{lem:flip}
    For all $f \in F(\Hexa)$, any translation-invariant ergodic Gibbs measure $\Pdimer$ is invariant under $\flip_f$, that is 
    $$\flip_f(\Pdimer) = \Pdimer.$$ 
    The measure obtained by performing a flip at $f$ with probability $1/2$ is also equal to $\Pdimer$.
\end{Lem}

This is well-known. 
In Appendix~\ref{app:dimers}, we explain how to deduce this from Theorem \ref{thm:toroidal:exhaustion}.

\begin{Rem}
	The flip operation and the Glauber dynamics that it naturally generates are classical in the dimer theory.
	A remarkable property (that we do not use) is the \emph{flip accessibility}, known since \cite{Th90}: on a simply connected finite subgraph of the hexagonal lattice with at least one dimer configuration, the Markov chain induced by the Glauber dynamic is irreducible (this is also true for subgraphs of the square lattice, with a different notion of a flip, see for example the review \cite{SalTom97}). In higher dimension, it is a long-standing open question whether the Markov chains induced by flips (and some necessary additional moves) are irreducible, see \cite{HLT23} for recent progress.
\end{Rem}

To be able to use flips, we need to know that a given hexagon does contain exactly three dimers with a positive probability under any non-frozen measure~$\Pdimer$. 
For a subset~$S \subset E(\Hexa)$ and a measure~$\Pdimer$ on dimer configurations, we denote by~$ \Pdimer(S)$ the probability that all edges of~$S$ are dimers.
\begin{Lem}\label{lem:hexagon}
    Let $\Pdimer$ be any translation-invariant non-frozen ergodic Gibbs measure and let $e_1,\dots,e_6$ be the consecutive edges of $\partial \fzero$. 
    Then,
    \begin{equation}
        \Pdimer(e_1,e_3,e_5) > 0.
    \end{equation}
\end{Lem}
This can be obtained from Theorem \ref{thm:KOS} after a few lines of computations: we give a more precise version of this result and its proof in Corollary~\ref{cor:hexagon:app}.

\subsection{Double dimers}\label{subsec:dd}
Given two subsets of edges $A, B \subset E(\Hexa)$, an \emph{alternated cycle} in $(A,B)$ is a cycle $(e_1, \dots, e_{2n})$ in $A \cup B$ such that, for all $1 \leq i \leq n$, $e_{2i} \in A$ and $e_{2i+1} \in B$.
An \emph{alternated infinite path} in $(A,B)$ is an infinite  path $(\dots, e_{-1},e_0, e_1,\dots)$ in $A \cup B$ such that for all $i \in \Z$, $e_{2i} \in A$ and $e_{2i+1} \in B$.
Each edge~$e\in A\cap B$ is called a \emph{doubled edge}.

\begin{Def}[Double dimers]
	The set of \emph{double dimer configurations} on $\Hexa$ is denoted by
	\[
	    \cD = \cD(\Hexa) = \{\ddim = \dim \cup \dim' \colon \dim \in \cM,~\dim' \in \cM\}. 
	\]
	If $\Pdimer$ denotes a Gibbs measure on $\cM$, a probability measure $\Pddimer$ is defined on $\cD$ as the law of
	$$
		\randomddim:= \randomdim \cup \randomdim',
	$$
    where $\randomdim, \randomdim'$ are two random independent dimer configurations with law $\Pdimer$.
\end{Def}
For $\dim,\dim' \in \cM$, the double dimer configuration $\dim \cup \dim'$ is a disjoint union of doubled edges, alternated cycles and alternated infinite paths in $(\dim,\dim')$ spanning all vertices of $\Hexa$. 

\begin{Rem}
	A double dimer configuration has a natural interpretation in terms of height functions: the difference $h_{\dim}-h_{\dim'}$ is a $3$-Lipschitz function on $F(\Hexa)$ taking values in $3\cdot \Z$, equal to zero at the origin and whose level lines are precisely the cycles and infinite paths in the double dimer configurations $\dim \cup \dim'$.
\end{Rem}

Conversely, we can sample two independent random dimer configurations using $\Pddimer$ as follows. The following lemma is classical and easy to prove.

\begin{Lem}\label{lem:sample}
	Let $\randomddim$ have the law $\Pddimer$.
	Now, define~$\randomdim$ and~$\randomdim'$ as follows: we put in both~$\randomdim$ and~$\randomdim'$ every edge that is doubled in~$\randomddim$; 
	we put in~$\randomdim$ every other edge of each cycle and infinite path independently for different cycles and paths, and we put in~$\randomdim'$ the remaining edges of these cycles and paths.
	Then, $\randomdim$ and $\randomdim'$ are two independent dimer configurations with law $\Pdimer$.
	In particular, a fully-packed loop configuration~$\randomloop$ defined as the complement of~$M$ has the law~$\Ploop$.
\end{Lem}

The main input we require concerning the double dimer model is that there exists infinitely many cycles around every point. We say that a cycle $\ell \subset E(\Hexa)$ \emph{surrounds} $X \subset \R^2$ if the continuous closed curve in $\R^2$ obtained by gluing together all edges of $\ell$ separates $X$ from infinity.

\begin{Thm}[Thm~4.8 of~\cite{KenOkoShe06}]\label{thm:rough}
    Let $\Pdimer$ be a non-frozen ergodic translation-invariant Gibbs measure on $\cM$, and $\Pddimer$ be the associated measure on $\cD$.
    Then, if $\randomddim$ has the law $\Pddimer$, there are almost surely infinitely many cycles in $\randomddim$ surrounding $\fzero$. In particular, $\randomddim$ almost surely contains no infinite paths.
\end{Thm}

    This statement holds for general $\Z^2$-periodic graphs with \emph{non-frozen} replaced by \emph{rough} (see Appendix \ref{app:dimers} for a definition).
    In Appendix~\ref{app:dimers}, we explain how~\cite[Theorem~4.8]{KenOkoShe06} simplifies to what we see in Theorem~\ref{thm:rough} in our particular case.
In short, this holds because on~$\Hexa$ non-frozen dimer measures are rough (see Corollary~\ref{cor:phase}).

\section{Structure of the proofs and easy implications}
For the rest of the article, we fix a translation-invariant non-frozen ergodic Gibbs measure $\Pdimer$. We denote by $\Ploop$ the measure on fully-packed loop configurations~$\Omfp$ obtained as complements to dimers (see Section~\ref{subsec:loop&dimers}).

\subsection{Topological description of the infinite clusters}\label{sec:notation:topology}
Recall the notations $\infclusters, \paths$ from Section~\ref{sec:notation:clusters}.
\begin{Def}\label{def:tree}
    The topological organization of the infinite paths in a configuration $\randomloop \in \Omfp$ can naturally be described using a tree $\tree(\randomloop)$, defined as follows (see Fig.~\ref{fig:path-swap}):
\begin{itemize}
    \item the vertices of $\tree(\randomloop)$ are the infinite clusters $\cluster \in \infclusters(\randomloop)$;
    \item the edges of $\tree(\randomloop)$ are the infinite paths $\gamma \in \paths(\randomloop)$;
    \item each edge connects the two clusters adjacent to it.
\end{itemize} 
\end{Def}
The tree $\tree(\randomloop)$ associated with a configuration is clearly invariant to translations.
Thus, by ergodicity of~$\Ploop$, this tree is almost surely a constant $\tree(\randomloop) = \tree$.
The number of edges $\Npath \in  \Z_{\geq 0} \cup \{\infty\}$ of $\tree$ is then almost surely a constant:
\begin{equation}\label{eq:def:Npath}
    \Ploop(|\paths(\randomloop)| = \Npath) = 1.
\end{equation}
Since $\tree$ is a tree, the number of vertices of $\tree$ is thus almost surely equal to $\Npath+1$:
\[
	\Ploop(|\infclusters(\randomloop)|= \Npath + 1) = 1.
\]
\begin{Def} 
A {\em trifurcation} (resp. {\em bifurcation})  is a vertex of $\tree$ of degree at least three (resp. exactly two).
Two infinite paths $\gamma_1, \gamma_2 \in \paths(\randomloop)$ are called {\em parallel} if the corresponding edges of~$\tree$ are incident to the same bifurcation vertex.
\end{Def}
One part of Theorem~\ref{thm:k:inf:2} is to prove that this tree has at most one edge, i.e., it is almost surely a single vertex or a single edge.
Note that when this is not the case, the tree always has at least one bifurcation or trifurcation.

We denote by $\tri$ the event that~$\tree$ has a trifurcation.
Note that by ergodicity, $\Ploop(\tri) \in \{0,1\}$.

\subsection{Main lemmata and derivation of Theorem~\ref{thm:k:inf:2}}\label{subsec:main results}



We now state several lemmata and then show how to derive Theorem~\ref{thm:k:inf:2} from them.

\begin{Lem}\label{lem:k:diff:2}
	The number of edges in $\tree$ is $\Npath \neq 2$.
\end{Lem}

\begin{Lem}\label{lem:k:inf:3}
	Assume that $\tree$ has at least $3$ edges i.e.,~$\Npath \in \Z_{\geq 3} \cup \{\infty\}$. Then, $\Ploop(\tri) > 0$. 
\end{Lem}

Note that any two~$\randomloop, \randomloop'\in \Omfp$ differing at finitely many edges have the same number of bi-infinite paths $\Npath$.
Thus, a standard approach in percolation theory via insertion tolerance does not apply in Lemma~\ref{lem:k:diff:2}.
To replace insertion tolerance, we developed a swapping technique to change the tree.
We describe the path swapping argument in Section~\ref{sec:swap} and then derive the lemmata in Section~\ref{sec:main:proof}.
Let us already provide a sketch of the argument.
We begin with the following slightly informal definition.
\begin{Def}
    In a configuration $\randomloop \in \Omfp$ with two infinite paths $\gamma_1, \gamma_2$, a {\em swap} of $\gamma_1$ and $\gamma_2$ is a local modification of $\randomloop$ into $\tilde{\randomloop} \in \Omfp$ rewiring the ends of $\gamma_1$ and $\gamma_2$.
This is pictured on top of Figure~\ref{fig:path-swap}. It has the following important property: if $\cluster$ is a bifurcation in $\tree(\randomloop)$ and $\gamma_1, \gamma_2$ are the two edges adjacent to $\cluster$, 
a swap of $\gamma_1$ and $\gamma_2$ modifies $\tree(\randomloop)$ into a new tree $\tree(\tilde{\randomloop})$.
This modification consists of merging together the three vertices adjacent to $\gamma_1$ and $\gamma_2$, and adding to this new vertex two edges $\tilde{\gamma}_1, \tilde{\gamma}_2$ towards new degree $1$ vertices.
This is pictured on the bottom of Figure~\ref{fig:path-swap}
\end{Def}

\begin{figure}
    \begin{center}
    	\includegraphics[width=0.8\textwidth]{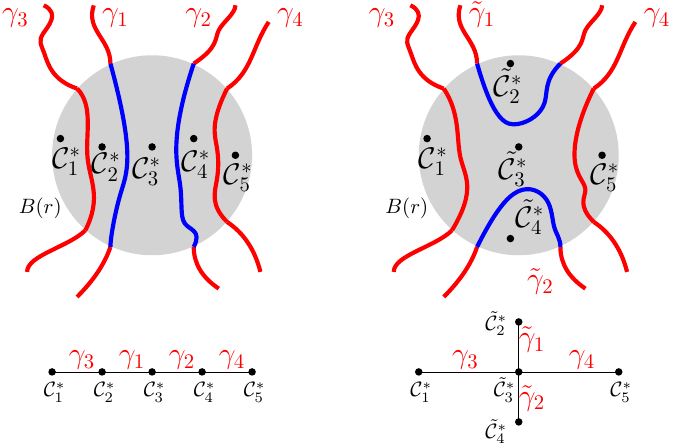}
    \end{center}
    \caption{A swap operation inside a ball~$B(r)$. 
    {\sc Top-Left:} a configuration of four infinite paths before the path swap. The swap will affect only the area between~$\gamma_1$ and~$\gamma_2$ inside~$B(r)$ (bounded by the two blue paths).
    {\sc Top-Right:} a configuration after the path swap. The infinite paths are now wired differently.
    {\sc Bottom:} the corresponding trees~$\tree$ on infinite dual clusters before and after the swap.
    }
    \label{fig:path-swap}
\end{figure}
Roughly speaking, the content of our main technical Lemma~\ref{lem:swap}, proven in Section~\ref{sec:swap}, is that it is always possible to perform such a swap using a bounded number of flips.

To deduce Lemma~\ref{lem:k:inf:3}, note that, if $\tree$ has at least $3$ edges, then either it almost surely has a trifurcation (and we are done) or it almost surely has a bifurcation $\cluster$.
In the latter case, the path swapping provides a trifurcation.
Since $\Ploop$ is invariant to the flip, this has positive probability.
The argument for Lemma~\ref{lem:k:diff:2} is very similar, but we have to use a local version of the swap since the path swapping of~$\tree$ in this case results in an isomorphic tree.

In addition to Lemmata~\ref{lem:k:diff:2} and~\ref{lem:k:inf:3}, we will also need the classical Burton-Keane argument~\cite{BurKea89} to rule out the existence of trifurcations:

\begin{Lem}\label{lem:no:tri}
    There is no trifurcation in $\tree$, i.e., $\Ploop(\tri) = 0$.
\end{Lem}

This statement is classical when some versions of shift invariance and amenability are available, e.g., see~\cite[Sections~$6,8$]{AldLyo07} for bond percolations on invariantly amenable unimodular graphs; and~\cite{LyoPer17} for site and bond percolation processes on transitive graphs.
Since some small modifications are needed in our context, we include a proof in Appendix~\ref{sec:burton-keane} for completeness.

We now show how to derive Theorem~\ref{thm:k:inf:2} from the lemmata stated above.
We use the XOR argument of~\cite{CraGlaHarPel20} and rely on the input from the double-dimer model provided in Theorem \ref{thm:rough}.

\begin{Lem}[\cite{CraGlaHarPel20}, Lemma 1.5]\label{lem:XOR:trick}
	Let~$\randomloop$ be a loop configuration without bi-infinite paths, i.e., a collection of non-intersecting cycles on~$\Hexa$.
	Take any~$r\in \N$.
	Then, for any cycle~$\Gamma$ that surrounds~$B(r)$, either~$\randomloop$ or~$\randomloop\oplus\Gamma$ has a loop of diameter at least~$r$ that surrounds~$\fzero$.
\end{Lem}

We provide a sketch of the proof of Lemma~\ref{lem:XOR:trick} for intuition (see also~\cite[Fig.~$3$]{CraGlaHarPel20}).
If all loops in~$\randomloop$ that surround~$\fzero$ are short, then none of them intersects~$\Gamma$.
XORing $\Gamma$ with loops that do not surround~$\fzero$ creates a new loop $\Gamma'$ (perturbation of~$\Gamma$) that surrounds~$\fzero$ and intersects~$\Gamma$, whence~$\Gamma'$ is long.

\begin{proof}[Proof of Theorem~\ref{thm:k:inf:2} assuming Lemmata~\ref{lem:k:diff:2}-\ref{lem:no:tri}]
	By Lemmata~\ref{lem:k:diff:2}-\ref{lem:no:tri}, $\Npath =0,1$.
	If~$\Npath=1$, we are done.
	Assume $\Npath = 0$ and take any~$r\in \N$. 
	By Theorem \ref{thm:rough}, $\Pddimer$-a.s. there is a cycle surrounding $B(r)$.
	Let~$\Gamma:= \{e_1, e_2, \dots, e_{2n}\}$ be the innermost such cycle.
	By Lemma~\ref{lem:sample}, a random sample~$\omega$ of~$\Ploop$ is obtained from a random sample~$D$ of~$\Pddimer$ via the following procedure: $\omega$ deterministically contains the complement $D^c$ and does not contain any double-edge of~$D$; we put in $\omega$ each other edge in each cycle and infinite path of~$D$ independently for different cycles and paths.
	In particular, flipping the state of all edges in~$\Gamma$ does not change the law of~$\omega$.
	The latter is called a XOR operation.
	In conclusion, the XOR with~$\Gamma$ is a measure-preserving operation on~$\Ploop$.
	Together with Lemma~\ref{lem:XOR:trick}, this implies 
	\[
		\Ploop(\randomloop \text{ contains a loop of diameter } \geq r \text{ surrounding } 0) \geq \tfrac{1}{2}.
	\]
	Since~$r\in\N$ is arbitrary and~$\Ploop$ is ergodic, we get the desired statement.
\end{proof}

\section{Path swapping lemma}\label{sec:swap}

Before stating the path swapping lemma, we need to introduce some notation regarding the geometry of loops, paths and dual clusters.

\paragraph{Loops and paths in finite balls.}
For a loop configuration~$\randomloop$, define~$\loops(\randomloop)$ as the set of its finite connected components.

\begin{Def}
    1) Let~$f \in F(\Hexa)$, $r,R\in\N$ satisfy~$r<R$. Consider a loop~$\ell = (e_0, \dots, e_k)$ with $e_i = x_ix_{i+1}$.
    The {\em branches of $\ell$ in $B(f,R) \setminus B(f,R)$} are the sub-sequences of $\ell$ such that
    \[
        (x_i,x_{i+1},\dots,x_{j},x_{j+1}) \sqsubset B(f,R) \setminus B(f,r), \quad x_i \in B(f,r), \quad x_{j+1} \not\in B(f,R).
    \]
    Recall that the notation $\gamma  \sqsubset X$ allows the endpoints of $\gamma$ not to be in $X$ (see Section~\ref{sec:notation}) so these three conditions are compatible.

    2) Now consider a bi-infinite path~$\gamma = (\dots,e_{-1},e_0,e_1,\dots)$ oriented from~$e_{-1}$ to~$e_1$ with $e_i = x_i x_{i+1}$.
    The branches of $\gamma$ in $B(f,R) \setminus B(f,r)$ can be defined as in the case of loops.
    Moreover, it will be convenient to define this also when~$R=\infty$:
    if~$\gamma$ intersects~$B(f,r)$, the {\em branches of $\gamma$ in $\R^2 \setminus B(f,r)$} are the two sub-sequences $(\dots, x_{i-1},x_i), (x_j,x_{j+1},\dots) \sqsubset B(f,r)^c$ such that~$x_i,x_j \in B(f,r)$.
    We denote them respectively by $\gamma_{f,R,-}$ and $\gamma_{f,R,+}$. 
\end{Def}
Note that the number of branches in $B(f,R) \setminus B(f,r)$ is always even.
For $f \in F(\Hexa)$, $r < R$ and $\randomloop \in \Omfp$, let 
\begin{equation}\label{eq:def:loops:paths:N}
    \begin{array}{ll}
    \loops_{f, R}(\randomloop) &:= \{\ell \in \loops(\randomloop),~\ell \cap B(f,R)  \neq \emptyset\},\\
    \paths_{f,R}(\randomloop) &:= \{\gamma \in \paths(\randomloop),~\gamma \cap B(f,R) \neq \emptyset\},\\
    N_{f,R}(\randomloop) &= |\paths_{f,R}(\randomloop)|.
    \end{array}
\end{equation}
For brevity, we drop $f$ from the notation when $f = \fzero$, for example we write
\[
	\loops_{R}(\randomloop) := \loops_{\fzero,R}(\randomloop) = \{\ell \in \loops(\randomloop),~\ell \cap B(R) \neq \emptyset\}.
\]
Observe that the notations $N_{f,R}$, $\paths_{f,R}$ coincide with the notations $\Npath$, $\paths$ of Section~\ref{sec:notation:topology} for the infinite paths when $R = \infty$.
For $f \in F(\Hexa)$, $r < R/2$, we define two events on which respectively \emph{loops are not too long} and \emph{paths are not too wiggly}: 
    \begin{align}
            \label{eq:def:nocrazyloop}
            \nocrazyloop(f,r,R) &:= \{\randomloop \in \Omfp \colon \forall \ell \in \loops_{f,R}(\randomloop) \,\, \ell \text{ has no branches in } B(f,\tfrac{R}{2})\setminus B(f,r)\},\\
            \label{eq:def:nocrazypath}
	        \nocrazypath(f,r,R) &:= \{\randomloop \in \Omfp \colon \forall \gamma \in \paths_{f,R}(\randomloop) \,\, \gamma \text{ has }\leq 2 \text{ branches in } B(f,\tfrac{R}{2})\setminus B(f,r)\}.
    \end{align}
    By translation-invariance, $\Ploop(\nocrazyloop(f,r,R))$ and $\Ploop(\nocrazypath(f,r,R))$ do not depend on $f$.
    Using monotonicity in $R$ and standard arguments, we obtain for all $f \in F(\Hexa)$, $r \in \N$,
    \begin{align}
        \label{eq:incr:Np}
        \Ploop(\Np_{f,R} \geq 1) &\xrightarrow[R \to \infty]{} \Ploop\Big(\bigcup_{R \in \Z_{>0}} (\Np_{f,R} \geq 1)\Big) = \Ploop(\Npath \geq 1)\\
        \label{eq:incr:loop}
        \Ploop(\nocrazyloop(f,r,R)) &\xrightarrow[R \to \infty]{} \Ploop\Big(\bigcup_{R \in \Z_{>r}} \nocrazyloop(f,r,R)\Big) = 1\\
        \label{eq:incr:paths}
        \Ploop(\nocrazypath(f,r,R)) &\xrightarrow[R \to \infty]{} \Ploop\Big(\bigcup_{R \in \Z_{>r}}\nocrazypath(f,r,R)\Big) =1.
	\end{align}

Before stating the path swapping lemma, we need some additional definitions.
\paragraph{Parallel and glued paths.}
An oriented infinite path $\gamma$ splits $\R^2$ into two connected subsets, $\R^2_{\gamma,+}$ and $\R^2_{\gamma,-}$ respectively on its right and left. The faces \emph{bordering $\gamma$ on the right} (resp. \emph{left}) are the faces $f \in F(\Hexa) \cap \R_{\gamma,+}^2$ (resp. $f \in F(\Hexa) \cap \R_{\gamma,-}^2$) such that $\partial f \cap \gamma \neq \emptyset$. 
If $\gamma_1$ and $\gamma_2$ are two disjoint oriented or unoriented infinite paths, the \emph{area} between $\gamma_1$ and $\gamma_2$ is defined as follows. Reorient each path such that the other one lies on its right. Then,
$$
\area(\gamma_1, \gamma_2) := \R^2_{\gamma_1,+} \cap \R^2_{\gamma_2,+}.
$$  

\begin{Def}
    Let $X \subset \R^2$. Two disjoint infinite paths $\gamma_1, \gamma_2$ are {\em glued in $X$} if
    \[
        \forall f\in F(\Hexa)\cap X\cap \area(\gamma_1, \gamma_2) \quad \partial f \cap \gamma_1 \neq \emptyset \text{ and } \partial f \cap \gamma_2 \neq \emptyset.
    \]
\end{Def}
Observe that $\gamma_1, \gamma_2 \in \paths(\randomloop)$ are parallel in $\randomloop$ if and only if
    \[
        \forall \gamma \in \paths(\randomloop) \setminus \{\gamma_1, \gamma_2\}, \quad
        \area(\gamma_1, \gamma_2) \cap \gamma = \emptyset.
    \]
Clearly, if $\gamma_1, \gamma_2$ are glued in $\R^2$, then they are parallel; the opposite is not always true.

\begin{Def}\label{def:dual:path}
    For $\randomloop \in \Omfp$, $\gamma_1, \gamma_2 \in \paths(\randomloop)$, a \emph{dual path in $\randomloop$ between $\gamma_1$ and $\gamma_2$} is a dual path $f_1, \dots, f_{k+1} \in F(\Hexa)$ such that $(f_1f_2)^{\star} \in \gamma_1$, $(f_kf_{k+1})^{\star} \in \gamma_2$ and $(f_if_{i+1})^{\star} \notin \randomloop$ for all $2 \leq i \leq k-1$.
\end{Def}
Note that any two parallel paths~$\gamma_1$ and~$\gamma_2$ are connected by a finite dual path.

Recall that for $X \subset \R^2$ and $\gamma = (f_1, \dots, f_{k+1})$, we write $\gamma \sqsubset X$ if $f_2, \dots, f_k \in F(\Hexa) \cap X$.
We emphasize that $f_1, f_{k+1}$ might not be in~$F(\Hexa) \cap X$. 
In particular, by our definitions, if $(f_1,\dots,f_{k+1})$ is a dual path between $\gamma_1$ and $\gamma_2$, then 
\[
	(f_1, \dots, f_{k+1}) \sqsubset \area(\gamma_1,\gamma_2) \quad \text{but} \quad 
	f_1, f_{k+1} \not\in \area(\gamma_1,\gamma_2) \cap F(\Hexa).
\]

\paragraph{The swapping lemma.}

Recall the flip operation from Definition~\ref{def:flip}.
The path swapping lemma states the following: 
given a ``typical'' $\randomloop \in \Omfp$ and a pair of parallel paths $\gamma_1, \gamma_2 \in \paths_R(\randomloop)$ separated by an infinite cluster $\cluster$ and oriented arbitrarily,
it is possible to rewire the four infinite branches $\gamma_{1,R,\pm}$ and $\gamma_{2,R,\pm}$ by performing a finite number of flips in $F(\Hexa) \cap B(R)$.
By this we mean that there exists $f_1, \dots, f_k \in  F(\Hexa) \cap B(R) $ such that~$\tomega := \flip_{f_k,\dots,f_1}(\randomloop)$
coincides with~$\randomloop$ outside of~$B(R)$ and
satisfies (see the top of Figure~\ref{fig:path-swap})
\begin{align}\label{eq:swap}
	\paths(\tomega)&= \paths(\randomloop)\setminus\{\gamma_1,\gamma_2\} \cup \{\tgamma_1,\tgamma_2\}, \\
	\text{ where } \tgamma_{i,R,-}&=\gamma_{i,R,-}, \tgamma_{i,R,+}=\gamma_{i,R,+} \text{ for } i=1,2. \nonumber
\end{align}

The key point is that this changes the topology $\tree$ of the infinite clusters as soon as $\Npath \geq 3$, see the bottom of Figure~\ref{fig:path-swap}.

For the above to hold, we first need to rule out certain pathological behavior.
Recall the definition of $\nocrazyloop(r,R)$ and $\nocrazypath(r,R)$ from Equations~\eqref{eq:def:nocrazyloop} and \eqref{eq:def:nocrazypath}.
On these events, respectively loops are not too long and paths are not too wiggly.

\begin{Lem}\label{lem:swap}
    Let $40 < 4r_0 < 2r_1 < R$. Let $\randomloop \in  \nocrazyloop(r_0,r_1) \cap \nocrazyloop(r_1,R)\cap \nocrazypath(r_1,R)$ and $\paths_{R}(\randomloop) = \{\gamma_1, \dots, \gamma_n\}$. Assume that~$\gamma_1, \gamma_2 \in \paths_{r_0}(\randomloop)$ and that these paths are parallel but not glued in $B(r_1)$.
    Then, for some integer $k \leq |F(\Hexa) \cap B(r_1)|+1$, there exists a sequence $g_1,\dots g_k \in B(R) \cap F(\Hexa)$ such that~$\tomega := \flip_{g_k, \dots, g_1}(\randomloop)$ satisfies Eq.~\eqref{eq:swap}.
\end{Lem}
The proof proceeds by induction in the distance between the paths~$\gamma_1$ and $\gamma_2$ and goes as follows:
\begin{itemize}
	\item if the paths are at distance one but are not glued, then we can find a flippable hexagon adjacent to both of them and this would do the swapping;
	\item if the paths are at distance more than one, then we can decrease the distance between them by flipping some hexagon between the paths and adjacent to~$\gamma_1$.
\end{itemize}

\begin{proof}
    By the properties of~$\randomloop$, $\gamma_1$ and~$\gamma_2$, we can always find a dual path $(f_1, f_2, \dots, f_{k+1}) \sqsubset B(r_1)$ in $\randomloop$ from $\gamma_1$ to $\gamma_2$.
    Proceed by induction on the minimal such $k \geq 2$:
    \begin{itemize}
    	\item {\em Base:} Assume that there exists $f \in F(\Hexa) \cap B(r_1)$ such that $\gamma_1 \cap \partial f \neq \emptyset$ and $\gamma_2 \cap \partial f \neq \emptyset$. Then, there exists $g \in F(\Hexa) \cap B(R)$ such that $\tomega = \flip_g(\randomloop)$ is the required path swapping.
    	\item {\em Induction step:} Assume $k\geq 3$ is the length of a minimal dual path $(f_1, \dots, f_{k+1}) \sqsubset B(r_1)$ from $\gamma_1$ to $\gamma_2$.
    Then, there exists $f \in F(\Hexa) \cap B(r_1)$ such that the following holds: $\randomloop' := \flip_f(\randomloop)$ satisfies~$\nocrazyloop(r_0,r_1)$, $\nocrazyloop(r_1,R)$, and~$\nocrazypath(r_1,R)$; the bi-infinite paths~$\gamma_1'$ and~$\gamma_2'$ (images of~$\gamma_1$ and~$\gamma_2$) intersect~$B(r_0)$ and are parallel in~$B(r_0)$ but not glued; there exists a dual path in~$B(r_1)$ between~$\gamma_1'$ and~$\gamma_2'$ of length strictly less than~$k$.
    \end{itemize}
    It is immediate that the combination of these two statements will imply the lemma. 
   
    \bigskip
    
	\emph{Proof of the base case.}
	Orient $\gamma_1$, $\gamma_2$ in such a way that each path is to the right of the other.
	Since these paths are not glued in~$B(r_1)$, there exists a face~$f'$ in~$F(\Hexa) \cap B(r_1)$ that borders one of the paths but not the other.
	Without loss of generality, we can assume that~$f'$ borders~$\gamma_1$.
	Let~$e\in \partial f\cap \gamma_1$, $e'\in \partial f'\cap \gamma_1$.
	Since~$\randomloop\in\nocrazyloop(r_1,R)$, once $\gamma_1$ exits $B(R/2)$ for the first time, it never returns in $B(r_1)$.
	Thus, the segment of $\gamma_1$ from $e$ to $e'$ remains in $B(R/2)$.
	Follow this segment and look at the faces bordering it on the right: denote by~$g$ the last face that borders~$\gamma_2$; denote by~$g'$ the face just after it (in the order along the segment).
	Since~$\gamma_1$ and~$\gamma_2$ are parallel and $g \in F(\Hexa) \cap \area(\gamma_1, \gamma_2)$, no other bi-infinite path in~$\randomloop$ borders~$g$.
	It remains to check that~$g$ is flippable and~$\tomega:=\flip_g(\randomloop)$ satisfies the required conditions.
	We verify this by considering all possible cases.
	
	Label the sides of $g$ (resp. $g'$) by $e_1, \dots, e_6$ (resp. $e'_1,\dots ,e'_6$) in counterclockwise order, with common edge $e_4 = e'_1$, and with $e_5 \in \gamma_1$, $e'_6 \in \gamma_1$.
    Since $\gamma_2 \cap \partial g \neq \emptyset$ and $\gamma_2 \cap \partial g' = \emptyset$, one of the $e_i$ is in $\gamma_2$ and none of the $e'_i$ is in $\gamma_2$. 
    We must have $e_6 \notin \gamma_2$ (because $\gamma_1 \neq \gamma_2$), $e_3 \notin \gamma_2$ (because otherwise $e'_1 \in \gamma_2$ or $e'_2 \in \gamma_2$, which are both excluded), and $e_2 \notin \gamma_2$ since otherwise the right extremity $u$ of $e_4 = e'_1$ would have degree $0$ in $\randomloop$. 
    The only remaining possibility is $e_1 \in \gamma_2$, but~$e_6, e_2, e_3, e_4 \notin \gamma_2$. Since $\randomloop$ is fully packed, $e_3$ must belong to some $\ell \in \loops(\randomloop)$ or some infinite path which can only be $\gamma_1$ or~$\gamma_2$ because $g \in F(\Hexa) \cap \area(\gamma_1, \gamma_2)$ and $\gamma_1$ and $\gamma_2$ are parallel.
    Thus, performing a flip at $g$ indeed rewires the infinite branches of $\gamma_1$ and $\gamma_2$ and does not affect the other paths.
    
    \bigskip
    
    \emph{Proof of the induction step.} 
    Denote by $(f_1, \dots, f_{k+1}) \sqsubset B(r_1)$ a dual path from $\gamma_1$ to $\gamma_2$ in $B(r_1)$ of minimal length.
    Denote by $u_0, u_1, \dots, u_5$ the vertices on $\partial f_2$, in clockwise order with $u_0u_1 = (f_1 f_2)^{\star} \in E(\Hexa)$.
    Consider the edges~$e_i := u_iu_{i+1}$ (with the convention that $u_6 := u_0$).\\
    
    {\bf Claim.}
    The face~$f_2$ is flippable and $e_2 \in \ell_1$, $e_4 \in \ell_2$ for distinct $\ell_1, \ell_2 \in \loops(\randomloop)$.
    
    \begin{proof}[Proof of claim]
		By minimality of $k$, $(f_2f_3)^{\star} = e_i$ for some $i \in \{2,3,4\}$.
        Indeed, if $(f_2f_3)^{\star} \in \{e_1,e_5\}$, then $(f_1, f_3, f_4, \dots, f_{k+1}) \sqsubset B(r_1)$ would be a dual path in $\randomloop$ between $\gamma_1$ and $\gamma_2$ of strictly smaller length. 
        Since $k \geq 3$, by definition of a dual path, we have $e_i \notin \randomloop$ and, thus, $e_{i-1}, e_{i+1} \in \randomloop$ (indeed, $\randomloop \in \Omfp$).
		Also, $e_{i-1}, e_{i+1} \notin \gamma_1$ by minimality of $k$.
		Indeed, otherwise there exists a face~$f'\not\in F(\Hexa) \cap \area(\gamma_1,\gamma_2)$ adjacent to~$f_2$ and~$f_3$, and we get a dual path  $(f', f_3, f_4, \dots, f_{k+1}) \sqsubset B(r_1)$ of length~$k-1$. We insist that $f'$ needs not be in $F(\Hexa) \cap B(r_1)$, due to our definition of $\sqsubset$. 

	    Together, the facts $e_{i-1}, e_{i+1} \in \randomloop$, $e_{i-1}, e_{i+1} \notin \gamma_1$ and $e_0 \in \gamma_1$ imply that $i=3$, i.e., $(f_2f_3)^{\star} = e_3 \notin \randomloop$ while $e_2, e_4 \in \randomloop \setminus \gamma_1$.
	    The latter also implies $e_1, e_5 \notin \randomloop$. In particular, $f_2$ is flippable. Furthermore, $e_2, e_4 \notin \gamma_2$ because we assumed that $k \geq 3$. Hence, $e_2 \in \ell_1$ and $e_4 \in \ell_2$ for some $\ell_1, \ell_2 \in \loops(\randomloop)$. 
	    The loops $\ell_1$ and $\ell_2$ are necessarily distinct. Indeed, assume for contradiction that they are the same loop $\ell = \ell_1 = \ell_2$. The loop $\ell$ contains $u_3$ and $u_4$, but not the edge $e_3 = u_3u_4$ because $e_3 \notin \randomloop$. Since the dual path $(f_2,\dots, f_{k+1})$ cannot use the dual edge $f_2f_3$ twice (by minimality), and since it intersects $\gamma_2$, it must also intersect an edge of $\ell$ (for topological reasons) which contradicts the definition of a dual path in $\randomloop$.
	\end{proof}
    
\begin{figure}
	\begin{center}
		\includegraphics[width=0.9\textwidth]{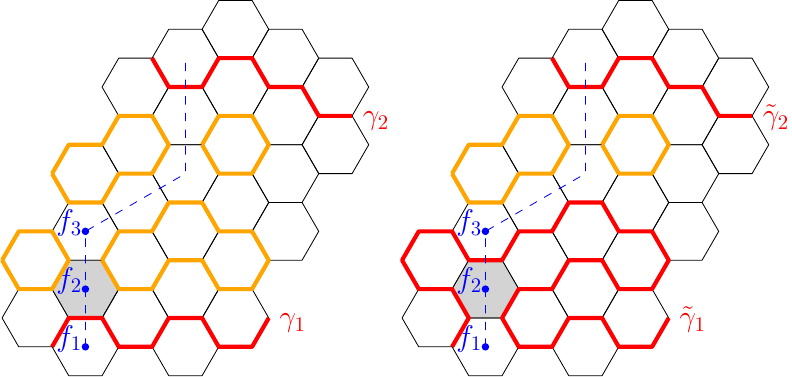}
	\end{center}
	\caption{An illustration of the induction step. The two infinite paths $\gamma_1$ and $\gamma_2$ are closer after perfoming a flip at $f_2$.}
    \label{fig:swap}
\end{figure}
    
    We come back to the proof of the induction step. Let $\omega' := \flip_{f_2}(\randomloop)$ (see Fig.~\ref{fig:swap}). Since the only infinite path in $\randomloop$ modified by the flip is $\gamma_1$, $\paths_{r_2}(\omega') = \{\gamma'_1, \gamma_2, \dots, \gamma_n\}$ where $\gamma'_1$ is obtained from $\gamma_1$ by gluing the loops $\ell_1$ and $\ell_2$ to $\gamma_i$ along $f_2$. More precisely, $\gamma'_1$ follows $\gamma_1$ until $u_0$, the edge~$u_0u_5$, $\ell_2$ counterclockwise from~$u_5$ to $u_4$, the edge $u_4u_3$, $\ell_1$ counterclockwise from~$u_3$ to $u_2$, the edge~$u_2u_1$ and then the rest of $\gamma_1$. 
    In particular, $e_3 = (f_2f_3)^{\star} \in \gamma'_1$, so $(f_2,\dots, f_{k+1}) \sqsubset B(r_1)$ is a dual path in $\omega'$ from $\gamma'_1$ to $\gamma_2$. 
    It remains to prove that $\omega'$ satisfies all the required properties:
    \begin{enumerate}[label=(\roman*)]
        \item $\loops(\omega') = \loops(\randomloop) \setminus \{\ell_1,\ell_2\}$, whence $\omega' \in \nocrazyloop(r_0,r_1) \cap \nocrazyloop(r_1,R)$;
        \item Since $\randomloop\in\nocrazyloop(r_1,R)$, neither $\ell_1$ nor $\ell_2$ intersects $\partial B(R/2)$.
        Also, by~$\randomloop\in\nocrazypath(r_1,R)$, the bi-infinite path~$\gamma_1$ has exactly two branches in $B(R/2) \setminus B(r_1)$.
        Then, the same holds for~$\gamma'_1$, which is a combination of~$\gamma_1$, $\ell_1$ and~$\ell_2$.
        This implies that~$\randomloop'\in \nocrazypath(r_1,R)$, as all other bi-infinite paths in~$\randomloop'$ coincide with those in~$\randomloop$.
        \item Since~$\gamma_1, \gamma_2$ intersect~$B(r_0)$, the same holds for~$\gamma'_1$ and~$\gamma'_2=\gamma_2$.
        Indeed, $\gamma_1\setminus\gamma'_1 = \{e_0\}$.
        If~$e_0$ is in~$B(r_0)$, then so is one of the four edges that share an endpoint with~$e_0$, which all belong to~$\gamma'_1$.
        \item $\gamma'_1$ and $\gamma'_2 = \gamma_2$ are parallel in $\randomloop'$, because $\area(\gamma'_1, \gamma_2) \subset \area(\gamma_1, \gamma_2)$ and all other bi-infinite paths of~$\randomloop$ are left unchanged by the flip at $f_2$.
        \item Choose any $f \in \big(F(\Hexa) \cap B(r_1) \cap \area(\gamma_1,\gamma_2)\big) \setminus \{f_2\}$ bordering $\gamma_1$. The face $f$ also borders $\gamma'_1$. Moreover, since $\randomloop$ does not satisfy the hypothesis of the base case, $\partial f \cap \gamma_2 = \emptyset$. This proves that $\gamma'_1$ and $\gamma_2$ are not glued in $B(r_1)$.
    \end{enumerate}
\end{proof}

In order to use Lemma \ref{lem:swap}, we need to show that bi-infinite paths are not glued together.
Note that if $\Pdimer$ were a frozen measure, the paths could in fact be glued. 
We believe that under non-frozen measures $\Ploop$, there exists almost surely no glued paths. Instead of proving this, we prove that when $\Ploop$ is non-frozen, even if some paths are glued together, they can be 'unglued' to apply Lemma \ref{lem:swap}.
To prove this, we use an input from dimer theory.
Define an event
\[
    \flippable = \{\randomloop \in \Omfp,~\fzero  \text{ is flippable in }\randomloop\}.
\]
By Lemma~\ref{lem:hexagon}, since~$\Ploop$ corresponds to a non-frozen measure~$\Pdimer$,
\[
    \Ploop(\flippable) > 0.
\]
We define one more event: for $r \in \R_+ \cup \{\infty\}$, let
\begin{align}\label{eq:def:glue:r}
	\glue_r := \{\randomloop \in \Omfp \colon \exists \gamma_1, \gamma_2 \in \paths_r(\randomloop)~ &\text{ not glued in } B(r) \text{ and adjacent in } \tree(\randomloop)\}.
\end{align}
This means that there exist two infinite paths intersecting $B(r)$, not glued in $B(r)$ and with a dual path in $\randomloop$ between them (Definition \ref{def:dual:path}).
We use the convention that $B(\infty) = \R^2$. 

\begin{Lem}\label{lem:unglue}
    Let $R> r> 0$.
    Assume that $\Npath \geq 2$ and take any~$\randomloop \in \nocrazyloop(r,R) \cap \flippable$ that satisfies~$\paths_r(\randomloop) \neq \emptyset$.
    Then, either $\randomloop \in \glue_{\infty}$ or $\flip_f(\randomloop) \in \glue_{\infty}$ for some $f \in F(\Hexa) \cap B(R)$.
\end{Lem}
\begin{proof}
    Denote by $e_0, \dots, e_5$ the sides of $\fzero $ in counterclockwise order. 
    Since the origin is flippable, we may assume, after renumbering, that $e_0, e_2, e_4 \in \randomloop$ and~$e_1, e_3, e_5 \notin \randomloop$.
    
    \bigskip
    
    \emph{Case 1: There exist $\gamma_1, \gamma_2 \in \paths(\randomloop)$ adjacent in $\tree(\randomloop)$ such that $\fzero  \in \area(\gamma_1, \gamma_2)$. In this case, we prove that $\randomloop \in \glue_{\infty}$.}
    
    Orient $\gamma_1$ and $\gamma_2$ such that each path is on the right of the other. 
    Assume for contradiction that $\gamma_1$ and $\gamma_2$ are glued. 
    We can assume that $e_0 \in \gamma_1$, $e_2 \in \gamma_2$.
    Consider the faces~$f,g \in F(\Hexa)$ such that~$(\fzero f)^{\star} = e_5$ and $(\fzero g)^{\star} = e_3$. 
    If $e_4 \in \gamma_1$, then~$f$ borders $\gamma_1$ on the right and has no edge in $\gamma_2$ (see the left of Fig.~\ref{fig:obstruction}), which is a contradiction. 
    The case $e_4 \in \gamma_2$ is analogous.
    
    Hence, $e_4 \in \ell \in \loops(\randomloop)$.
    This implies that~$f,g \in F(\Hexa) \cap \area(\gamma_1, \gamma_2)$. 
    Since we assumed that~$\gamma_1$ and~$\gamma_2$ are glued, $\partial f \cap \gamma_2 \neq \emptyset$ and~$\partial g \cap \gamma_1 \neq \emptyset$. This is impossible for topological reasons (see the right of Fig.~\ref{fig:obstruction}). Thus, the paths are not glued in $\R^2$.
    \begin{figure}
    \centering
	\includegraphics[width=0.25\textwidth, page=1]{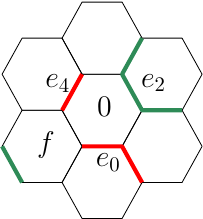}
	\hspace{0.1\textwidth}
	\includegraphics[width=0.25\textwidth, page=2]{topological_obstruction.pdf}
    \caption{Illustration of case 1 in the proof of Lemma \ref{lem:unglue}: $\gamma_1$ is drawn in red, $\gamma_2$ in green $\ell$ in blue. Both situations lead to a topological obstruction: it is impossible to complete these drawings by joining together the segments of each color.}
    \label{fig:obstruction}
    \end{figure}
    
    \bigskip
    
    \emph{Case 2: For any~$\gamma_1, \gamma_2 \in \paths(\randomloop)$ adjacent in $\tree(\randomloop)$, $\fzero  \notin F(\Hexa) \cap \area(\gamma_1, \gamma_2)$. In this case, we prove that there exists $f \in F(\Hexa) \cap B(R)$ such that $\flip_f(\randomloop) \in \glue_{\infty}$.}
    
    There exists $\gamma_1 \in \paths(\randomloop)$ which can be oriented in such a way that $\fzero  \in F(\Hexa) \cap \R^2_{-, \gamma_1}$ and, for any~$\gamma \in \paths(\randomloop) \setminus \{\gamma_1\}$, we have~$\gamma \subset \R^2_{+, \gamma_1}$.
    Our assumption~$\paths_r(\randomloop) \neq \emptyset$ implies that~$\gamma_1 \cap B(r) \neq \emptyset$.
    Denote by $\cluster \in \infclusters(\randomloop)$ the unique infinite cluster bordering $\gamma_1$ on the right.
    Consider any path~$\gamma_2 \in \paths(\randomloop)\setminus\{\gamma_1\}$, such that $\gamma_1$ and $\gamma_2$ bound a common cluster (such path indeed exists since~$\Npath \geq 2$).
    
    Assume that we can find a flippable hexagon $g \in F(\Hexa) \cap B(R)$ bordering $\gamma_1$ \emph{on the left}.
    Then, $\tomega = \flip_g(\randomloop)$ has the same infinite paths as $\randomloop$ except that $\gamma_1$ is replaced by $\tgamma_1$. The paths $\tgamma_1, \gamma_2 \in \paths(\tomega)$ still bound a common cluster $\tcluster \in \infclusters(\tomega)$. Moreover $g$ borders $\tgamma_1$ \emph{on the right} so $g \in F(\Hexa) \cap \area(\tgamma_1, \gamma_2) \cap B(R)$ and $\partial g \cap \gamma_2 = \emptyset$. This proves that $\tgamma_1$ and $\gamma_2$ are not glued, thus $\tomega \in \glue_\infty$.
    
    It remains to prove that such face~$g$ indeed exists.
    If $\partial \fzero  \cap \gamma_1 \neq \emptyset$, then~$\fzero $ borders $\gamma_1$ on the left and is flippable, and the claim follows.
    Otherwise, $e_0, e_2, e_4$ belong to some loops in $\loops(\randomloop)$ since $\fzero \in \R_{-,\gamma_1}^2$ and no infinite path intersects $\R_{-,\gamma_1}^2$. In particular, there is at least one loop $\ell_0 \in \loops_r(\randomloop)$ on the left of $\gamma_1$. 
    
    We claim that we can find $\ell \in \loops_r(\randomloop)$ and $f \in F(\Hexa) \cap B(r)$ bordering $\gamma_1$ on the left such that $\partial f \cap \ell \neq \emptyset$. 
    Let us prove this. 
    Since $\ell_0 \cap B(r) \neq \emptyset$ and $\gamma_1 \cap B(r) \neq \emptyset$, we can find a face path $f_1, \dots, f_{n+1}$ in $F(\Hexa) \cap B(r)$ from $\ell_0$ to $\gamma_1$. If $\partial f_1 \cap \gamma_1 \neq \emptyset$, the claim follows since by definition $\partial f_1 \cap \ell_0 \neq \emptyset$. Otherwise, let us denote by $1 \leq k \leq n$ the maximal index such that $\partial f_k \cap \gamma_1 = \emptyset$. Let $e_0 = (f_k f_{k+1})^*$ and $e_1 \in \partial f_{k+1}$ sharing a vertex with $e_0$. 
    Note that $\partial f_k \cap \gamma_1 = \emptyset$ implies $e_0, e_1 \notin \gamma_1$. 
    Since $\randomloop$ is fully packed, at least one of the two edges $e_0,e_1$ must be in $\randomloop$. 
    There are no infinite paths on the left of $\gamma_1$, so $e_0 \in \ell$ or $e_1 \in \ell$ for some $\ell \in \loops_r(\randomloop)$. In either case, the claim is established. 
    
    Due to the assumption~$\randomloop \in \nocrazyloop(r,R)$, we have $\ell \cap \partial B(R) = \emptyset$. Hence, we can find two hexagons $g, \tg \in F(\Hexa) \cap B(R)$ consecutive bordering~$\gamma_1$ on the left such that $\partial g \cap \ell \neq \emptyset$, $\partial \tg \cap \ell = \emptyset$. 
    In more detail, start from $f$ and follow $\gamma_1$. Define~$g$ as the last hexagon bordering $\gamma_1$ on the left having a side in $\ell$; define~$\tg$ as the next hexagon after~$g$.
    Considering all cases proves that $g$ is flippable, which concludes the proof. 
    The argument is omitted here as it is identical to the base case of Lemma \ref{lem:swap}.
\end{proof}

\section{End of the proof: deducing Lemmata \ref{lem:k:diff:2} and \ref{lem:k:inf:3} from path swapping}\label{sec:main:proof}

\begin{proof}[Proof of Lemma \ref{lem:k:inf:3}]
	The strategy outlined after the lemma’s statement is now formalized.
	Assume for contradiction that $\Ploop(\tri)=0$ and take any~$\randomloop \notin \tri$. 

	\emph{Step~$1$: ungluing.} 
	To apply the swapping argument, we must check the condition that the paths are not glued.
	Recall that $\flippable$ denotes the event on which the origin $\fzero $ is flippable. By Lemma \ref{lem:hexagon}, 
	\[
		\eps := \Ploop(\flippable) > 0.
	\]
    Recall from \eqref{eq:def:loops:paths:N}, \eqref{eq:def:nocrazyloop} and \eqref{eq:def:nocrazypath} the definitions of $\Np_r(\randomloop)$, $\nocrazyloop(r,R)$ and $\nocrazypath(r,R)$. 
    By Equations~\eqref{eq:incr:Np} and \eqref{eq:incr:loop}, we can fix $r \in \Z_{>0}$ and $R > r$ large enough such that 
	\[
        \Ploop(\Np_r \geq 1) \geq 1-\eps/3 \quad \text{and} \quad \Ploop(\nocrazyloop(r,R)) \geq 1-\eps/3.
	\]
	Hence, by the union bound,
	\[
		\Ploop(\flippable \cap (\Np_r \geq 1) \cap \nocrazyloop(r,R)) \geq \eps/3 > 0.
	\]
    For any~$\randomloop \in \flippable \cap (\Np_r \geq 1) \cap \nocrazyloop(r,R)$, Lemma \ref{lem:unglue} applies: either~$\randomloop\in \glue_{\infty}$ or~$\tomega:=\flip_f(\randomloop)\in\glue_{\infty}$ for some~$f \in F(\Hexa) \cap B(R)$.
    The operation~$\flip_f$ is a measure-preserving operation and, for each~$\tomega$, the number of preimages is bounded by~$|F(\Hexa) \cap B(R)|$.
    This implies that $\Ploop(\glue_{\infty})>0$:
    \begin{align*}
        \eps/3 
        &\leq \Ploop(\flippable \cap (\Np_r \geq 1) \cap \nocrazyloop(r,R))\\
        &\leq \Ploop\Bigg(\glue_{\infty} \cup \bigcup_{f \in F(\Hexa) \cap B(R)} \flip_f^{-1}(\glue_{\infty})\Bigg)\leq (|F(\Hexa)\cap B(R)|+1)\Ploop(\glue_{\infty}).
    \end{align*}
    
    \bigskip
    
    \emph{Step~$2$: the swapping.} We are now in a position to apply the path swapping lemma. The events $\glue_r$ defined in \eqref{eq:def:glue:r} are increasing in $r$ so
    \begin{equation}
    	\label{eq:incr:glue}
        \Ploop(\glue_r) \xrightarrow[r\to\infty]{} \Ploop\left(\bigcup_{r \in \Z_{>0}}\glue_r\right) = \Ploop(\glue_{\infty}).
    \end{equation}
    By the second step, $\Ploop(\glue_{\infty}) > 0$ so we can fix $r_0 > 10$ such that $\delta := \Ploop(\glue_{r_0}) > 0$. By \eqref{eq:incr:loop} and \eqref{eq:incr:paths}, we can fix one after the other $R > 2r_1 > 4r_0$ large enough such that 
    \[
        \min(\Ploop(\nocrazyloop(r_0,r_1)),\Ploop(\nocrazyloop(r_1,R)),\Ploop(\nocrazypath(r_1,R)))\geq 1-\delta/4.
    \]
    Let
    \[
        \conditions := \glue_{r_0} \cap \nocrazyloop(r_0,r_1) \cap \nocrazyloop(r_1,R) \cap \nocrazypath(r_1,R).
    \]
    By the union bound,
    \begin{equation} \label{eq:conditions}
        \Ploop(\conditions) \geq \delta/4 > 0.
    \end{equation}
    Recall our assumptions~$\Npath \geq 3$ and~$\Ploop(\tri)=0$ and that~$\Npath < \infty$ by Lemma~\ref{lem:no:tri}. Then, almost surely, $\tree$ is a finite path.
	Denote this event by~$\contradict$. Thus,
    \[
        \Ploop(\{\Npath \geq 3\} \cap \contradict \cap \conditions) \geq \frac{\delta}{4}>0.
    \]
    Let $\randomloop \in \{\Npath \geq 3\} \cap \contradict \cap \conditions$. 
    We check that conditions of Lemma \ref{lem:swap} are satisfied: some of them are implied by~$\conditions$; by definition of $\glue_{r_0}$, there exist $\gamma_1, \gamma_2 \in \paths_{r_0}(\randomloop)$ adjacent in $\tree$ which are not glued in $B(r_0)\subset B(r_1)$; finally, by~$\contradict$, the paths $\gamma_1$ and $\gamma_2$ are parallel.
    Hence, by Lemma~\ref{lem:swap}, there exist faces $f_1, \dots, f_k$ in $B(R) \cap F(\Hexa)$ with $k \leq |F(\Hexa)\cap B(r_1)|+1 =: n(r_1)$ such that~$\tomega:=\flip_{f_1,\dots, f_k}(\randomloop)\not\in \contradict$.
    Similarly to Step~$1$, we use that the~$\flip$ operation is measure-preserving to obtain~$\Ploop(\contradict)<1$:
    \begin{align*}
        \delta/4 \leq \Ploop(\{\Npath \geq 3\} \cap \contradict \cap \conditions) 
        &\leq \Ploop\Bigg(\bigcup_{\substack{f_1,\dots, f_k \in (F(\Hexa)\cap B(R))\\ k \leq n(r_1)}}\flip_{f_1,\dots, f_k}^{-1}(\contradict^{c})\Bigg)\\
        &\leq n(r_1)|F(\Hexa)\cap B(R)|^{n(r_1)}\cdot (1-\Ploop(\contradict)).\qedhere
    \end{align*}
\end{proof}

\begin{proof}[Proof of Lemma~\ref{lem:k:diff:2}]
    The idea of the proof is as follows.
    When $\conditions(\fzero)$ and $\conditions(f)$ happen simultaneously with $|\fzero - f| \geq \cR$, applying path swapping (Lemma \ref{lem:swap}) both around $\fzero$ and $f$ creates a loop of length $\cR$ intersecting $B(R)$, with probability bounded from below by some constant $c(R)>0$, uniformly in~$\cR$.
    Letting~$\cR \to \infty$ then yields a contradiction.
    
    Let us make this formal.
    We reason by contradiction, assuming that 
    \begin{equation}\label{eq:contradict:N=2}
        \Ploop(\Npath=2)=1.
    \end{equation}
    We let $\paths(\randomloop) = \{\gamma_1, \gamma_2\}$. 
    As in the proof of Lemma~\ref{lem:k:inf:3}, fix $R > 2r_1 > 4r_0 > 10$ such that \eqref{eq:conditions} holds.
    For $f = k + \ell e^{i\pi/3} \in F(\Hexa)$ with $k,\ell \in \Z$, we denote by $\conditions(f)$ the translation of the event $\conditions$ by $k + \ell e^{i\pi/3}$. By translation invariance, $\Ploop(\conditions(f)) = \Ploop(\conditions) > 0$ for all $f \in F(\Hexa)$.
    Let $k \in \N$ such that 
    \be\label{eq:choice:N}
        k > 2\Ploop(\conditions)^{-1}.
    \ee
    Let $\cR > 2R$.
    If we pick any $f_1, ..., f_k \in F(\Hexa)$ such that $|f_i - f_j| \geq \cR$ for all $i \neq j$, then by inclusion-exclusion 
    \be
        \begin{aligned}
        1 \geq \Ploop\left(\bigcup_{i=1}^k \conditions(f_i)\right) &\geq \sum_{i = 1}^k \Ploop(\conditions(f_i)) - \sum_{i \neq j} \Ploop(\conditions(f_i) \cap \conditions(f_j))\\
        & \overset{\eqref{eq:choice:N}}{\geq} 2 - \sum_{i \neq j} \Ploop(\conditions(f_i) \cap \conditions(f_j)).
        \end{aligned}
    \ee
    Hence, there exist $i \neq j$ with 
    \be
        \Ploop(\conditions(f_i) \cap \conditions(f_j)) \geq \frac{2}{k(k-1)} > 0.
    \ee
    By translation invariance, there exists $f$ with $|f-\fzero| \geq \cR$ such that
    \be
        \Ploop(\conditions(\fzero) \cap \conditions(f)) \geq \frac{2}{k(k-1)} > 0.
    \ee
    Note that $k$ \emph{does not depend on $\cR$} and that the face $f$ is fixed (it does not depend on $\randomloop$). Since we assumed that \eqref{eq:contradict:N=2} holds, it implies 
    $$
        \Ploop((\Npath =2)\cap\conditions(\fzero) \cap \conditions(f)) \geq \frac{2}{k(k-1)} > 0.
    $$
    Applying path swapping Lemma~\ref{lem:swap} both around $0$ and $f$ creates a loop of length at least $L - 2R$ that intersects $B(R)$, see Figure~\ref{fig:create:loop}.
    \begin{figure}\centering
        \begin{overpic}[abs,unit=1mm,scale=.5]{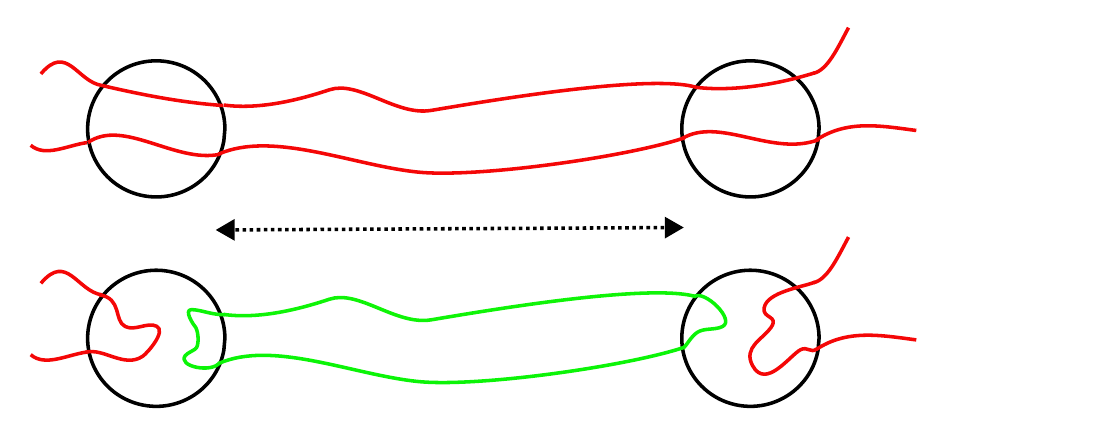}
        \put(9,34){$B(r_3)$}
        \put(58,34){$B(x,r_3)$}
        \put(37,15){$R$}
        \put(-2,31){\textcolor{red}{$\gamma_1$}}
        \put(-2,25){\textcolor{red}{$\gamma_2$}}
        \put(0,14){\textcolor{red}{$\tilde{\gamma}_1$}}
        \put(79,9){\textcolor{red}{$\tilde{\gamma}_2$}}
        \put(37,2){\textcolor{green}{$\ell$}}
        \end{overpic}
        \caption{On the top: the two infinite paths intersecting simultaneously two boxes at distance at least $R$. On the bottom: the situation after the flip moves in both boxes, a long loop $l$ is created.}
        \label{fig:create:loop}
    \end{figure}
    Using that $\Ploop$ is invariant by flips as in the proof of the preceding lemma, we obtain that
    \be
        \Ploop(\exists \ell \in \loops_{R}(\randomloop), |\ell| \geq \cR) \geq c(R) \frac{2}{k(k-1)}.  
    \ee
    for some $c(R)>0$ depending only on $R$. Since the events on the left-hand side are decreasing in $\cR$,
    \be
        \Ploop(\forall \cR > 2R, \exists \ell \in \loops_R(\randomloop), |\ell| \geq R) \geq c(R) \frac{2}{k(k-1)} > 0.
    \ee
    This is a contradiction: since all loops are finite, this event has probability $0$.
\end{proof}

\begin{Rem}
    Note that this proof also applies to exclude any $\Npath \in \Z_{\geq 2}$ (but not $\Npath = \infty$).
\end{Rem}

\appendix

\section{Fundamental results on dimers from \cite{She05} and \cite{KenOkoShe06}}\label{app:dimers}
In this Appendix, we recall the general theorems of \cite{She05} and \cite{KenOkoShe06} in a particular setting of the hexagonal lattice. 
In this section, we will denote a face $f \in F(\Hexa)$ by its coordinates.
The center of any face $f$ can be written uniquely under the form $k\cdot 1 + l \cdot e^{\pi i/3}$, in which case we identify $f = (k,l)$. 
Recall the height function $h_\dim$ associated with a dimer configuration $\dim$ (Definition~\ref{def:height:function}).
If $f = (k,l) \in F(\Hexa)$, we write $h_{\dim}(f) = h_m(k,l)$.
Let $\Pdimer$ be a translation-invariant Gibbs measure and let $\E$ denote the expectation under $\Pdimer$. The \emph{slope} $(s,t) \in \R^2$ of $\Pdimer$ is defined by
\[
	s = \E[h_\dim(-1,0)] \quad \text{and} \quad  t = \E[h_\dim(0,1)].
\]
For $k \in \N$, let $\cM_k$ denote the set of dimer configurations which are invariant to the translations by $k$ and $k\cdot e^{\pi i/3}$.
For $\dim \in \cM_k$, the \emph{height change} $(h_\dim^1,h_\dim^2)$ is defined by
\[
	(h_\dim^1,h_\dim^2):= \big(h_\dim(-k,0), h_\dim(0,k)\big).
\]
Note that even though $\dim$ is periodic, $h_\dim$ is not; hence the height change is a non-trivial object. For $i,j \in \Z^2$, define
$$
	\cM_k(i,j) := \{\dim \in \cM_k, (h_\dim^1,h_\dim^2) = (i,j)\}.
$$
The following result is due to Sheffield \cite{She05}, but we give the version of \cite{KenOkoShe06}.
\begin{Thm}[\cite{She05}]\label{thm:toroidal:exhaustion}
   The slope $(s,t)$ of a Gibbs measure $\Pdimer$ is an element of the polygon $3(\NewtonPolygon-(1,1))$ where $\NewtonPolygon$ 
   is the triangle with vertices $(0,0),(0,1),(1,0)$. 
   Reciprocally, for $(s,t) \in 3(\NewtonPolygon-(1,1))$, the weak limit $\Pdimer$ of the uniform measures on $\cM_k(\lfloor sk \rfloor, \lfloor tk \rfloor)$ exists and is the unique translation-invariant ergodic Gibbs measure on $\cM$ with slope $(s,t)$.
   
   Moreover, for $(s,t) \in 3(\NewtonPolygon-(1,1))$, the unique translation-invariant ergodic Gibbs measure of slope $(s,t)$ is non-frozen if and only if $(s,t)$ lies in the interior of $3(\NewtonPolygon-(1,1))$.
\end{Thm}

\begin{Rem}
    For a general bipartite $\Z^2$-periodic graph with periodic weights, the set of slopes is a translate of the Newton polygon (see e.\,g. \cite[Section~$3.1.4$]{KenOkoShe06} for a definition). Our definition of the height function is $3$ times that of \cite{KenOkoShe06}, hence the multiplicative factor $3$ for the set of allowed slopes.  The translation depends on the choice of the reference flow in the definition of the height function. The correct translation is obtained as a byproduct of the following proof.
\end{Rem}

    The first part of the statement is proven in~\cite[Proposition~$3.2$, Theorem~$2.1$]{KenOkoShe06}.
    We explain how to obtain the last statement of the theorem concerning frozen measures since this is not directly stated in \cite{KenOkoShe06}: they use a definition of a frozen measure involving the height function which is less natural in the context of the loop $O(1)$ model.
    The argument that we use here is taken from~\cite[Section~$4.2.1$]{KenOkoShe06}. 
\begin{proof}
	Let~$\Pdimer$ be a translation-invariant ergodic Gibbs measure on dimers with slope $(s,t) \in 3(\NewtonPolygon-(1,1))$.
	We need to show that~$\Pdimer$ is frozen if and only if $(s,t)$ lies on the boundary of $3(\NewtonPolygon-(1,1))$. By Definition \ref{def:height:function}, for any dimer configuration $\dim$, we have 
    \[
        -1 \leq h_\dim(-1,0) \leq 2 \quad ; \quad -1 \leq h_\dim(0,1) \leq 2 \quad ; \quad -2 \leq h_\dim(-1,1) \leq 1.
    \]
    Taking expectations (and using translation-invariance), we obtain
    \[
        -1 \leq s \leq 2 \quad ; \quad -1 \leq t \leq 2 \quad ; \quad -2 \leq t+s \leq 1.
    \]
    In other words, $s$ and $t$ belong to the triangle $3(\NewtonPolygon-(1,1))$. In particular, $(s,t)$ lies on the boundary of $3(\NewtonPolygon-(1,1))$ if and only if one of these inequalities is an equality. 
    
    If none of these inequalities is an equality, then, for all $e \in E(\Hexa)$, we have $0 < \Pdimer(e) < 1$, whence $\Pdimer$ is non-frozen. 
    
    If at least one of these inequalities is an equality, this means that there exists an edge $e$ with $\Pdimer(e) \in \{0,1\}$.
    By translation-invariance, either (i) all translates of~$e$ are in~$\dim$ almost surely or (ii) none of the translates of~$e$ is in~$\dim$ almost surely.
    By Definition~\ref{def:frozen}, the corresponding loop measure~$\Ploop$ (and hence~$\Pdimer$)  is frozen: if (i) occurs, then~$\Ploop$ is concentrated on one configuration given by removing all translates of~$e$ from~$\Hexa$; if (ii) occurs, then all translates of~$e$ belong to the loop configuration almost surely.
\end{proof}

Using Theorem \ref{thm:toroidal:exhaustion}, it is easy to see why Lemma \ref{lem:flip} holds.
\begin{proof}[Proof of Lemma \ref{lem:flip}]
    The uniform measures on $\cM_k(\lfloor sk \rfloor, \lfloor tk \rfloor)$ mentioned in Theorem \ref{thm:toroidal:exhaustion} are invariant to the flip. Hence, so are their limits, which parameterize the set of all translation-invariant ergodic Gibbs measures.
\end{proof}

The set of translation-invariant ergodic Gibbs measures on bipartite graphs has an alternative, more explicit parametrization: this is the object of the next theorem. 
Recall that the set of vertices is partitioned into $WV$ and $BV$. 
Given a weight function $\rho: E(\Hexa) \to \R_{>0}$, the associated \emph{weighted Kasteleyn matrix} $K$ of the hexagonal lattice  is the weighted adjacency matrix with rows and columns indexed by~$w\in WV$ and~$b\in BV$:
\begin{equation}
	\label{eq:def:kasteleyn}
    K(w,b) :=
    \begin{cases}
    	\rho(wb) & \text{if }w \sim b \\
    	0 & \text{otherwise.}
    \end{cases}
\end{equation}
For $A,B,C \geq 0$, we denote by $K_{(A,B,C)}$ the Kasteleyn matrix associated with the weight function $\rho$ defined on pairs of adjacent $w \in WV$ and~$b\in BV$ by 
\[
    \rho(wb) = 
    A \cdot \1_{(b-w)\sqrt{3} = i} +  
    B \cdot \1_{(b-w)\sqrt{3} = e^{7\pi i/6}} + 
    C \cdot \1_{(b-w)\sqrt{3} = e^{-\pi i/6}}.
\]
\begin{Rem}
    In the general Kasteleyn theory \cite{Kas67}, the entries of~$K$ are signed with respect to a \emph{Kasteleyn orientation}.
    In such an orientation, for each face, the number of its sides oriented counterclockwise is odd.
    Orienting all edges from white to black provides a natural periodic Kasteleyn orientation on~$\Hexa$ and simplifies~$K$ to what we see in~\eqref{eq:def:kasteleyn}.
\end{Rem}

Kenyon, Okounkov and Sheffield give an explicit description of the set of all translation-invariant ergodic Gibbs measures on dimer configurations of $\Z^2$-periodic graphs. 
In the case of the hexagonal lattice, this description is given by the following theorem. 

\begin{Thm}[\cite{KenOkoShe06}]\label{thm:KOS}
    The set of all translation-invariant ergodic Gibbs measures of Theorem \ref{thm:toroidal:exhaustion} can be alternatively parameterized as $\{\Pdimer_{(A,B,C)},~A,B,C > 0\}$. For all $A,B,C > 0$, the following explicit expression holds: for $\{w_1b_1,\dots,w_nb_n\} \subset E(\Hexa)$,
    \[
        \Pdimer_{(A,B,C)}(w_1b_1,\dots,w_nb_n) = \left(\prod_{i=1}^n K_{(A,B,C)}(w_i,b_i)\right) \det(K_{(A,B,C)}^{-1}(b_i,w_j)_{1\leq i,j \leq n})
    \]
    where $K^{-1}_{(A,B,C)}$ is defined for all $w \in WV, b \in BV$ by
    \begin{equation}\label{eq:K:inverse}
        K^{-1}_{(A,B,C)}(b,w) = \frac{1}{(2i\pi)^2}\int_{|z_1|=|z_2|=1}\frac{z_1^{[wb]_1} z_2^{[wb]_2}}{A+\frac{B}{z_2}+Cz_1}\cdot\frac{dz_1}{z_1}\cdot\frac{dz_2}{z_2}.
    \end{equation}
    with
    \[
        w - b = \tfrac{1}{\sqrt{3}}e^{\pi i/6}+ [wb]_1 \cdot e^{\pi i /3} + [wb]_2 \quad ; \quad [wb]_1, [wb]_2 \in \Z^2.
    \]
    The measure $\Pdimer_{(1,1,1)}$ is the weak limit of the uniform measure on the set $\cM_k$ of dimer configurations invariant to the translations by $k\cdot 1$ and $k \cdot e^{i\pi/3}$. 
\end{Thm}
Note that since $K_{(A,B,C)}$ is an infinite matrix, the existence of its inverse is not obvious, so here we take~\eqref{eq:K:inverse} as a definition of $K_{(A,B,C)}^{-1}$.
\begin{Rem}
    1) The measures $\Pdimer_{(A,B,C)}$ are obtained as weak limits of the measures on $\cM_k$ attributing a different weight $A,B,C$ to the three different types of edges modulo translation.

    2) The parametrizations of Theorems \ref{thm:toroidal:exhaustion} and Theorem \ref{thm:KOS} can be related via the concepts of \emph{surface tension} and \emph{Ronkin function} which we will not introduce here (see \cite[Sections~$3.2.3$ and~$3.2.4$]{KenOkoShe06} for a definition).
\end{Rem}

A corollary of the explicit expression of $\Pdimer_{(A,B,C)}$ is the following.
\begin{Cor}[\cite{KenOkoShe06}]\label{cor:frozen}
    For $A,B,C > 0$, the translation-invariant ergodic Gibbs measure $\Pdimer_{(A,B,C)}$ is non-frozen if and only if $A,B,C$ satisfy the strict triangle inequality.
\end{Cor}
This is obtained by evaluating explicitly $K^{-1}_{(A,B,C)}$ for the hexagonal lattice. 
If $\theta_A, \theta_B, \theta_C$ are the angles of the triangle with sides $A,B,C$, and if $e_1,e_2,e_3$ are the three types of edges of the hexagonal lattice (up to translation), computing directly edge probabilities gives
$$
    \Pdimer(e_1) = \frac{\theta_A}{\pi} \quad ; \quad \Pdimer(e_2) = \frac{\theta_B}{\pi} \quad ; \quad \Pdimer(e_3) = \frac{\theta_C}{\pi}.
$$
See~\cite[Section 6.1]{Ken09} for more details.

Another corollary of Theorem~\ref{thm:KOS} is the following. A translation-invariant ergodic non-frozen measure $\Pdimer$ is called \emph{smooth} if the variance of the height difference under $\Pdimer$ between two points is bounded by a constant independent of the distance and it is called \emph{rough} otherwise. 

\begin{Cor}[\cite{KenOkoShe06}]\label{cor:phase}
    Any non-frozen translation-invariant ergodic Gibbs measure $\Pdimer$ is rough (no smooth phase).
\end{Cor}

The height difference between two faces can be recovered from $K^{-1}_{(A,B,C)}$.
It can be shown that for any $\Pdimer = \Pdimer_{(A,B,C)}$ with $A,B,C > 0$ satisfying the strict triangle inequality, the height difference between two faces grows logarithmically with the distance, see~\cite[Section 6.3]{Ken09}. 
This corollary is also a direct consequence of the general phase diagram of $\Z^2$-periodic dimer models  established in~\cite[Section 4]{KenOkoShe06}. 

This explains why~\cite[Theorem~4.8]{KenOkoShe06} simplifies to what we see in Theorem~\ref{thm:rough}.

\begin{proof}[Proof of Theorem~\ref{thm:rough}]
	Consider any non-frozen translation-invariant ergodic Gibbs measure~$\Pdimer$ for dimers on~$\Hexa$.
	By Corollary~\ref{cor:phase}, it is rough.
	By~\cite[Theorem~4.8]{KenOkoShe06}, for any rough ergodic Gibbs measure on a $\Z^2$-periodic bipartite graph, the associated double dimer law $\Pddimer$ exhibits almost surely infinitely many alternated cycles around every point.
\end{proof}

Another corollary of Theorem~\ref{thm:KOS} is the following. 
\begin{Cor}\label{cor:hexagon:app}
    Let $A,B,C > 0$ satisfy a strict triangle inequality. Let $\theta_A, \theta_B, \theta_C >0$ denote the angles of the triangle with side lengths $A,B,C$ and $e_1,\dots,e_6$ be the consecutive edges of $\fzero$.
    Then,
    \begin{equation}\label{eq:dimers-explicit}
    	\Pdimer(e_1,e_3,e_5) = \frac{1}{\pi^3}
        \det
        \begin{pmatrix}
            -\sin(\theta_A) & \theta_B & \theta_C \\
            \theta_B & -\sin(\theta_C) & \theta_A \\
            \theta_C & \theta_A & -\sin(\theta_C)
        \end{pmatrix}.
    \end{equation}
    In particular, this probability is strictly positive, because $\sin(\theta)<\theta$ for any~$\theta>0$.
\end{Cor}
The case~$\Pdimer_{(1,1,1)}$ was treated in~\cite[Section 4.6]{Ken97};
in general, the operator $K^{-1}_{(A,B,C)}$ can be evaluated explicitly as in \cite[Section~$6$]{Ken09} and the same computations as in~\cite[Section 6.3]{Ken09} replacing $\Pdimer_{(A,B,C)}$ with $\Pdimer_{(1,1,1)}$ give the formula.

We now deduce Lemma~\ref{lem:hexagon}, which states that, under any non-frozen translation-invariant ergodic Gibbs measure~$\Pdimer$ for dimers on~$\Hexa$, a given hexagon contains exactly three dimers with a positive probability.
\begin{proof}[Proof of Lemma~\ref{lem:hexagon}]
    By Theorem~\ref{thm:KOS} and Corollary~\ref{cor:frozen}, if $\Pdimer$ is a non-frozen translation-invariant ergodic Gibbs measure for dimers on~$\Hexa$, $\Pdimer = \Pdimer_{(A,B,C)}$ for some~$A,B,C > 0$ that satisfy a strict triangle inequality.
    Hence Lemma~\ref{lem:hexagon} follows from Corollary~\ref{cor:hexagon:app}.
\end{proof}

\section{Ruling out trifurcation points: Burton--Keane argument}
\label{sec:burton-keane}
Our proof of Lemma~\ref{lem:no:tri} below closely follows~\cite[Section 8.2]{Gri99a}.
In particular, we rely on~\cite[Lemma~8.5]{Gri99a} that proves a general statement about $3$-partitions.
A {\em $3$-partition} $\Pi = \{P_1,P_2,P_3\}$ of a finite set $S$ with $|S| \geq 3$ is a partition of $S$ into three distinct non empty sets $P_1,P_2,P_3$. Two $3$-partitions $\{P_1,P_2,P_3\}$ and $\{P_1',P_2',P_3'\}$ are {\em compatible} if there exists an ordering of their elements such that $P_2' \cup P_3' \subset P_1$, or equivalently $P_2 \cup P_3 \subset P_1'$. A family of $3$-partitions of a set $S$ is compatible if any two distinct $3$-partitions in the family are compatible.

\begin{Lem}[\cite{Gri99a}, Lemma 8.5]\label{lem:grimmett}
	If $\cP$ is a compatible family of $3$-partitions of a finite set $S$ with $|S| \geq 3$, then $|\cP| \leq |S|-2$.
\end{Lem}

\begin{proof}[Proof of Lemma~\ref{lem:no:tri}]
    For $f \in F(\Hexa)$, $r > 0$, let $\cluster_r(f)$ be the connected component of $f$ in the restriction of $\cluster(f)$ to $B(r)$.
    We say that $f$ is an \emph{$r$-coarse trifurcation point} if $\cluster(f) \setminus \cluster_r(f)$ has at least three distinct infinite connected components.
    Let us denote this event by $\tricoarse(f,r)$. 
    By translation-invariance, $\Ploop(\tricoarse(f,r))$ does not depend on $f$.
    We prove that $\Ploop(\tri) = 0$ by contradiction, in three steps:
    \begin{itemize}
        \item \emph{Step 1}: if $\Ploop(\tri) > 0$, then there exists $r$ large enough such that $\Ploop(\tricoarse(f,r)) > 0$;
        \item \emph{Step 2}: if $f$ is an $r$-coarse trifurcation and $B(f,r) \subset B(R)$, $f$ defines a $3$-partition of $\cluster(f) \cap \partial B(R)$ two such $3$-partitions are compatible whenever~$|f-f'| \geq R$;
        \item \emph{Step 3}: we obtain a contradiction by Lemma~\ref{lem:grimmett}.
    \end{itemize}

    \emph{Step 1:}
    Assume for contradiction that $\Ploop(\tri) > 0$. 
    For $f \in F(\Hexa)$, define the event~$\tri(f)$:
    $$
    	\tri(f) = \{\cluster(f) \text{ has degree $3$ in $\tree$}\}.
    $$
    By translation-invariance of~$\Ploop$ and since~$\Ploop(\tri) > 0$, we get that~$\Ploop(\tri(f)) = p >0$ does not depend on~$f \in F(\Hexa)$. For $r > 0$, define: 
    $$
        \tri(f,r) := \{\text{$\cluster(f)$ has $3$ neighbors $\clusterbis_i$ in $\tree$ such that}\colon
     	\forall i \in \{1,2,3\},~\clusterbis_i \cap B(f,r) \neq \emptyset\}.
    $$
    For a given~$f \in F(\Hexa)$, the events $\tri(f,r)$ are increasing in $r$ and $\bigcup_{r \in \Z_{>0}} \tri(f,r) = \tri(f)$.
    Hence, by translation-invariance, there exists $r_0 \in \N$ large enough such that for all $f \in F(\Hexa)$, $\Ploop(\tri(f,r_0)) \geq p/2$.
    Since $\Ploop(\tri)>0$, $\Ploop(\Npath \geq 1)=1$ so by Equations~\eqref{eq:incr:loop} and~\eqref{eq:incr:paths} we can fix $r$ such that $\Ploop(\nocrazyloop(f,r_0,r) \cap \nocrazypath(f,r_0,r)) \geq 1-p/4$. By union bound, 
    \begin{equation*}
        \Ploop(\tri(f,r_0) \cap \nocrazyloop(f,r_0,r) \cap \nocrazypath(f,r_0,r)) \geq p/4.
    \end{equation*}
   	We conclude Step~$1$ by showing the following inclusion of events:
    $$
        \tricoarse(f,r) \supset \tri(f,r_0) \cap \nocrazyloop(f,r_0,r) \cap \nocrazypath(f,r_0,r),
    $$
    Indeed, the $\clusterbis_i$ from the definition of $\tri(f,r_0)$ break $\cluster(f)\setminus B(f,r_0)$ into at least three distinct infinite connected components. 
    Moreover, on the event $\nocrazyloop(f,r_0,r) \cap \nocrazypath(f,r_0,r)$, it holds that $\cluster(f) \cap B(f,r_0) \subset \cluster_r(f)$, so $\cluster(f) \setminus \cluster_r(f) \subset \cluster(f) \setminus B(f,r_0)$ also has at least three infinite connected components.
    
    \bigskip

    \emph{Step 2:} 
    Let $R > 0$, $f \in F(\Hexa)$ be such that $B(f,r) \subset B(R)$ and $f$ is an $r$-coarse trifurcation.
    Let us denote by $(A_i(f))_{1 \leq i \leq n(f)}$ the connected components of $\cluster(f) \setminus \cluster_r(f)$ intersecting $B(R)$. 
    Partition $\{1,\dots, n(f)\}$ arbitrarily in three non-empty subsets, $\{1,\dots, n(f)\} = I_1(f) \sqcup I_2(f) \sqcup I_3(f)$ and define
    $$
        P_k(f) = \bigsqcup_{i \in I_k(f)} A_i(f) \cap \partial B(R) \quad ; \quad k \in \{1,2,3\}.
    $$
    By definition of $\tri(f,r)$, at least three of the $A_i(f)$ are infinite so we can choose the partition such that $P_k(f) \neq \emptyset$ for $k \in \{1,2,3\}$. 
    This defines a $3$-partition
    $$
        \cluster(f) \cap \partial B(R) = P_1(f) \sqcup P_2(f) \sqcup P_3(f).
    $$
    Let us now prove that if $B(f,r) \cap B(f',r) = \emptyset$, the $3$-partitions $(P_1(f),P_2(f),P_3(f))$ and $(P_1(f'),P_2(f'),P_3(f'))$ are compatible.
    Since $B(f,r) \cap B(f',r) = \emptyset$, we have~$\cluster_r(f) \cap \cluster_r(f') = \emptyset$.
    Let us consider the connected components in $\cluster(f) \setminus (\cluster_r(f) \cup \cluster_r(f'))$ intersecting $B(R)$, and split them into three groups:

    (a) the components adjacent only to $\cluster_r(f)$, let us denote them by $B_i$

    (b) the components adjacent only to $\cluster_r(f')$, let us denote them by $C_i$

    (c) the components adjacent to both $\cluster_r(f)$ and $\cluster_r(f')$, let us denote them by $D_i$.
    
    Since $\cluster_r(f')$ is connected, all $C_i$ and $D_i$ belong to the same connected component of $\cluster(f) \setminus \cluster_r(f)$, hence upon renumbering the $A_i(f)$,
    $$
        A_1(f) = \left(\bigcup C_i\right) \cup \left(\bigcup D_i\right) \cup \cluster_r(f') \quad ; \quad A_i(f) = B_{i-1},~2 \leq i \leq n(f).
    $$
    Similarly, upon renumbering the $A_i(f')$ 
    $$
        A_1(f') = \left(\bigcup B_i\right) \cup \left(\bigcup D_i\right) \cup \cluster_r(f) \quad ; \quad A_i(f') = C_{i-1},~ 2 \leq i \leq n(f').
    $$
    Hence, the $3$-partitions associated to $f$ and $f'$ are compatible: upon renumbering the $P_k(f)$ and $P_k(f')$, we can assume that $1 \in I_1(f)$, $1 \in I_1(f')$ and then
    \begin{align*}
        P_1(f) 
        &\supset [\partial B(R) \cap A_1(f)]
        \supset \partial [B(R) \cap \bigcup C_i]\\
        &= [\partial B(R) \cap \bigcup_{2 \leq i \leq n(f')} A_i(f')]
        \supset [P_2(f') \cup P_3(f')].
    \end{align*}

     \bigskip
    
    \emph{Step 3:} recall that $r$ was fixed in Step~$1$.
    For $(k, \ell) \in \Z^2$, let $f_{k,\ell} \in F(\Hexa)$ denote the face centered at $2r(k + \ell e^{i \pi/3})$. 
    Note that $B(f_{k,\ell},r) \cap B(f_{k',\ell'},r) = \emptyset$ whenever~$(k,\ell)\neq (k',\ell')$.
    Let $R>0$, $\randomloop \in \Omfp$ and $\cluster \in \infclusters(\randomloop)$ be any infinite cluster intersecting $B(R+r)$, if such exists.
    Given any set of faces~$F$, define~$t(F)$ as the set of $f_{k,l} \in F$ such that $\randomloop \in \tri(f_{k,l},r)$.
    By Steps~$1$ and~$2$, $t(\cluster \cap B(R))$ induces a compatible family of $3$-partitions of $\cluster \cap \partial B(R+r)$.
    Hence, by Lemma \ref{lem:grimmett}
       \[
        |t(\cluster \cap B(R))| \leq |\cluster \cap \partial B(R+r)|.
    \]
    Summing over all infinite clusters intersecting $B(R+r)$, we obtain
    \begin{equation}\label{eq:tri:upper:bound}
        |t(B(R)\cap F(\Hexa))| \leq |\partial B(R+r)\cap F(\Hexa)| \leq c(r)R,
    \end{equation}
	for some constant $c(r)$ that depends only on $r$. 
	On the other hand,
	\[
		|\{(k,\ell) \in \Z^2, f_{k,\ell} \in B(R)\cap F(\Hexa) \}| \geq R^2/(4r^2),
	\]
    and every $f_{k,\ell} \in F(\Hexa) \cap B(R)$ is a trifurcation point with probability at least~$p/2$.
    Thus,
    \[
        \Ploop[|t(B(R)\cap F(\Hexa))|] \geq C(r)p(r)R^2,
    \]
    This contradicts \eqref{eq:tri:upper:bound} for $R$ large enough.
\end{proof}

\newcommand{\etalchar}[1]{$^{#1}$}
\providecommand{\bysame}{\leavevmode\hbox to3em{\hrulefill}\thinspace}
\providecommand{\MR}{\relax\ifhmode\unskip\space\fi MR }
\providecommand{\MRhref}[2]{%
  \href{http://www.ams.org/mathscinet-getitem?mr=#1}{#2}
}
\providecommand{\href}[2]{#2}


\begin{thebibliography}{CDH{\etalchar{+}}14}

\bibitem[ABF87]{AizBarFer87}
M.~Aizenman, D.~J. Barsky, and R.~Fern{{\'a}}ndez, \emph{The phase
  \mbox{transition} in a general class of {I}sing-type models is sharp}, J.
  Statist. Phys. \textbf{47} (1987), no.~3-4, 343--374. \MR{894398 (89f:82013)}

\bibitem[Agg23]{Agg23}
Amol Aggarwal, \emph{Universality for lozenge tiling local statistics}, Annals
  of Mathematics \textbf{198} (2023), 881--1012.

\bibitem[Aiz80]{Aiz80}
M.~Aizenman, \emph{Translation invariance and instability of phase coexistence
  in the two-dimensional {I}sing system}, Comm. Math. Phys. \textbf{73} (1980),
  no.~1, 83--94.

\bibitem[AL07]{AldLyo07}
David~J Aldous and Russell Lyons, \emph{Processes on unimodular random
  networks.}, Electronic Journal of Probability [electronic only] \textbf{12}
  (2007), 1454--1508.

\bibitem[BDS25]{BerDaS25}
Nathana{\"e}l Berestycki and William Da~Silva, \emph{Critical behaviour of the
  fully packed loop-${O}(n)$ model on planar triangulations}, arXiv preprint
  arXiv:2512.05867 (2025).

\bibitem[BH19]{BenHon19}
St{\'e}phane Benoist and Cl{\'e}ment Hongler, \emph{The scaling limit of
  critical ising interfaces is {CLE}3}, The Annals of Probability \textbf{47}
  (2019), no.~4, 2049--2086.

\bibitem[BK89]{BurKea89}
R.~M. Burton and M.~Keane, \emph{Density and uniqueness in percolation}, Comm.
  Math. Phys. \textbf{121} (1989), no.~3, 501--505.

\bibitem[BN89]{BloNie89}
Henk~W.J. Bl{\"o}te and Bernard Nienhuis, \emph{The phase diagram of the {O}(n)
  model}, Physica A: Statistical Mechanics and its Applications \textbf{160}
  (1989), no.~2, 121 -- 134.

\bibitem[CDH{\etalchar{+}}14]{CheDumHonKemSmi14}
Dmitry Chelkak, Hugo {Duminil-Copin}, Cl\'{e}ment Hongler, Antti Kemppainen,
  and Stanislav Smirnov, \emph{Convergence of {I}sing interfaces to {S}chramm's
  {SLE} curves}, C. R. Math. Acad. Sci. Paris \textbf{352} (2014), no.~2,
  157--161. \MR{3151886}

\bibitem[CDIV14]{CoqDumIofVel14}
Loren Coquille, Hugo {Duminil-Copin}, Dmitry Ioffe, and Yvan Velenik, \emph{On
  the {G}ibbs states of the noncritical {P}otts model on {$\Bbb{Z}^2$}},
  Probab. Theory Related Fields \textbf{158} (2014), no.~1-2, 477--512.

\bibitem[CGHP25]{CraGlaHarPel20}
Nicholas Crawford, Alexander Glazman, Matan Harel, and Ron Peled,
  \emph{Macroscopic loops in the loop {O}(n) model via the {XOR} trick}, The
  Annals of Probability \textbf{53} (2025), no.~2, 478--508.

\bibitem[CKP01]{CohKenPro01}
Henry Cohn, Richard Kenyon, and James Propp, \emph{A variational principle for
  domino tilings}, Journal of the American Mathematical Society \textbf{14}
  (2001), no.~2, 297--346.

\bibitem[CN06]{CamNew06}
Federico Camia and Charles~M. Newman, \emph{Two-dimensional critical
  percolation: the full scaling limit}, Comm. Math. Phys. \textbf{268} (2006),
  no.~1, 1--38.

\bibitem[CS11]{CheSmi11}
Dmitry Chelkak and Stanislav Smirnov, \emph{Discrete complex analysis on
  isoradial graphs}, Adv. Math. \textbf{228} (2011), no.~3, 1590--1630.

\bibitem[DRT19]{DumRaoTas19}
Hugo {Duminil-Copin}, Aran Raoufi, and Vincent Tassion, \emph{Sharp phase
  transition for the random-cluster and {P}otts models via decision trees},
  Ann. of Math. (2) \textbf{189} (2019), no.~1, 75--99.

\bibitem[DT16]{DumTas15}
H.~{Duminil-Copin} and V.~Tassion, \emph{A new proof of the sharpness of the
  phase transition for {B}ernoulli percolation and the {I}sing model},
  Communications in {M}athematical {P}hysics \textbf{343} (2016), no.~2,
  725--745.

\bibitem[Fis66]{Fis66}
M.~Fisher, \emph{On the dimer solution of planar {I}sing models}, Journal of
  Math. Physics \textbf{7} (1966), no.~10, 1776--1781.

\bibitem[GHZ25]{GlaHarZel25}
Alexander Glazman, Matan Harel, and Nathan Zelesko, \emph{Planar percolation
  and the loop o(n) model}, arXiv preprint arXiv:2508.20917 (2025).

\bibitem[GM23]{GlaMan23}
Alexander Glazman and Ioan Manolescu, \emph{Structure of {G}ibbs measures for
  planar {FK}-percolation and {P}otts models}, Probability and Mathematical
  Physics \textbf{4} (2023), no.~2, 209--256.

\bibitem[Gri99]{Gri99a}
G.~Grimmett, \emph{Percolation}, second ed., Grundlehren der Mathematischen
  Wissenschaften [Fundamental Principles of \mbox{Mathematical} Sciences], vol.
  321, Springer-Verlag, Berlin, 1999.

\bibitem[Hig81]{Hig81}
Y.~Higuchi, \emph{On the absence of non-translation invariant {G}ibbs states
  for the two-dimensional {I}sing model}, Random fields, {V}ol. {I}, {II}
  ({E}sztergom, 1979), Colloq. Math. Soc. J{\'a}nos Bolyai, vol.~27,
  North-Holland, Amsterdam, 1981, pp.~517--534.

\bibitem[HLT24]{HLT23}
Ivailo Hartarsky, Lyuben Lichev, and Fabio~Lucio Toninelli, \emph{Local dimer
  dynamics in higher dimensions}, Annales de l’Institut Henri Poincar{\'e} D
  (2024).

\bibitem[Jac99]{Jac99}
Jesper~Lykke Jacobsen, \emph{On the universality of fully packed loop models},
  Journal of Physics A: Mathematical and General \textbf{32} (1999), no.~29,
  5445.

\bibitem[Kas61]{Kas61}
P.~W. Kasteleyn, \emph{The statistics of dimers on a lattice}, PhysicA
  \textbf{27} (1961), 1209--1225.

\bibitem[Kas67]{Kas67}
Pieter~Willem Kasteleyn, \emph{Graph theory and crystal physics}, Graph Theory
  Theor. Phys. (1967), 43--110.

\bibitem[Ken97]{Ken97}
Richard Kenyon, \emph{Local statistics of lattice dimers}, Ann. Instit. Henri
  Poincare Probab. Stat. \textbf{33} (1997), no.~5, 591--618.

\bibitem[Ken01]{Ken01}
Richard Kenyon, \emph{Dominos and the {G}aussian free field}, Ann. Probab.
  \textbf{29} (2001), no.~3, 1128--1137.

\bibitem[Ken02]{Ken02}
R.~Kenyon, \emph{The {L}aplacian and {D}irac operators on critical planar
  graphs}, Invent. Math. \textbf{150} (2002), no.~2, 409--439. \MR{1933589
  (2004c:31015)}

\bibitem[Ken09]{Ken09}
Richard Kenyon, \emph{Lectures on dimers}, Statistical mechanics, IAS/Park City
  Math. Ser., vol.~16, Amer. Math. Soc., Providence, RI, 2009, pp.~191--230.
  \MR{2523460 (2010j:82023)}

\bibitem[KN04]{KagNie04}
Wouter Kager and Bernard Nienhuis, \emph{A guide to stochastic {L}{\"o}wner
  evolution and its applications}, J. Statist. Phys. \textbf{115} (2004),
  no.~5-6, 1149--1229.

\bibitem[KOS06]{KenOkoShe06}
Richard Kenyon, Andrei Okounkov, and Scott Sheffield, \emph{Dimers and
  amoebae}, Ann. of Math. (2) \textbf{163} (2006), no.~3, 1019--1056.
  \MR{2215138 (2007f:60014)}

\bibitem[KS21]{KhrSmi21}
M.~Khristoforov and S.~Smirnov, \emph{Percolation and ${O}(1)$ loop model},
  arXiv:2111.15612 (2021).

\bibitem[Las18]{Las19}
Beno{\^{\i}}t Laslier, \emph{Local limits of lozenge tilings are stable under
  bounded boundary height perturbations}, Probability Theory and Related Fields
  \textbf{173} (2018), no.~3-4, 1243--1264.

\bibitem[LP17]{LyoPer17}
Russell Lyons and Yuval Peres, \emph{Probability on trees and networks},
  vol.~42, Cambridge University Press, 2017.

\bibitem[New90]{New90}
Charles~M. Newman, \emph{Ising models and dependent percolation}, Topics in
  statistical dependence ({S}omerset, {PA}, 1987), IMS Lecture Notes Monogr.
  Ser., vol.~16, Inst. Math. Statist., Hayward, CA, 1990, pp.~395--401.
  \MR{1193993}

\bibitem[Nie82]{Nie82}
Bernard Nienhuis, \emph{Exact critical point and critical exponents of {${\rm
  O}(n)$} models in two dimensions}, Phys. Rev. Lett. \textbf{49} (1982),
  no.~15, 1062--1065.

\bibitem[PS19]{PelSpi17}
Ron Peled and Yinon Spinka, \emph{Lectures on the spin and loop {O}(n) models},
  Springer Proceedings in Mathematics and Statistics (2019), 246--320.

\bibitem[She05]{She05}
Scott Sheffield, \emph{Random surfaces}, Ast\'{e}risque (2005), no.~304,
  vi+175. \MR{2251117 (2007g:82021)}

\bibitem[Smi01]{Smi01}
Stanislav Smirnov, \emph{Critical percolation in the plane: conformal
  invariance, {C}ardy's formula, scaling limits}, C. R. Acad. Sci. Paris
  S\'{e}r. I Math. \textbf{333} (2001), no.~3, 239--244.

\bibitem[Smi10]{Smi10}
\bysame, \emph{Conformal invariance in random cluster models. {I}.
  {H}olomorphic fermions in the {I}sing model}, Ann. of Math. (2) \textbf{172}
  (2010), no.~2, 1435--1467.

\bibitem[ST96]{SalTom97}
Nicolau Saldanha and Carlos Tomei, \emph{An overview of domino and lozenge
  tilings}, Resenhas do Instituto de Matem{\'a}tica e Estat{\'i}stica da
  Universidade de S{\~a}o Paulo \textbf{2} (1996), no.~2, 239--252.

\bibitem[Tas16]{Tas16}
Vincent Tassion, \emph{Crossing probabilities for {V}oronoi percolation}, Ann.
  Probab. \textbf{44} (2016), no.~5, 3385--3398.

\bibitem[Thu90]{Th90}
William~P. Thurston, \emph{Conway's tiling groups}, Am. Math. Monthly
  \textbf{97} (1990), no.~8, 757--756.

\end{thebibliography}
\end{document}